\documentclass{iopart}

\pdfoutput=1

\usepackage[utf8]{inputenc}
\usepackage{iopams}
\expandafter\let\csname equation*\endcsname\relax
\expandafter\let\csname endequation*\endcsname\relax
\usepackage{amsmath}
\usepackage{amsthm}
\usepackage{cite}
\usepackage{bm}
\usepackage{graphicx}
\usepackage{subcaption}
\usepackage{color}
\usepackage{hyperref}
\usepackage{paralist} 
\usepackage[margin=1in]{geometry}

\graphicspath{{figs/}}

\DeclareMathOperator{\Li}{Li}

\newtheorem{lemma}{Lemma}
\newtheorem{theorem}{Theorem}

\newtheorem{corollary}{Corollary}
\newtheoremstyle{examplestyle}  
    {\topsep}                    
    {\topsep}                    
    {}                           
    {}                           
    {\bfseries}                  
    {.}                          
    {.5em}                       
    {}  
\theoremstyle{examplestyle}

\newcommand{\mathnotation}[2]{\newcommand{#1}{\ensuremath{#2}}}

\newcommand{\keepnote}[1]{}

\mathnotation{\ldef}{\mathrel{\raisebox{.069ex}{:}\!\!=}}
\mathnotation{\cc}{c}                        
\mathnotation{\CC}{C}
\newcommand{\E}{\mathbb{E}}
\newcommand{\pmat}[4]{\begin{pmatrix} #1 & #2 \\ #3 & #4\end{pmatrix}}
\mathnotation{\M}{M}
\mathnotation{\K}{K}

\begin{document}

\title{Lyapunov exponents for the random product of two shears}

\author{Rob Sturman}
\address{Department of Applied Mathematics, University of Leeds, Leeds LS2
  9JT, UK}
\ead{r.sturman@leeds.ac.uk}
\author{Jean-Luc Thiffeault}
\address{Department of Mathematics, University of Wisconsin --
  Madison, \\ 480 Lincoln Dr., Madison, WI 53706, USA}
\ead{jeanluc@math.wisc.edu}

\date{\today}

\begin{abstract}
We give lower and upper bounds on both the Lyapunov exponent and
generalised Lyapunov exponents for the random product of positive and
negative shear matrices. These types of random products arise in
applications such as fluid stirring devices.  The bounds,
obtained by considering invariant cones in tangent space, give
excellent accuracy compared to standard and general bounds, and are
increasingly accurate with increasing shear.  Bounds on generalised
exponents are useful for testing numerical methods, since these
exponents are difficult to compute in practice.
\end{abstract}



\section{Introduction}

Random matrix products have applications in many disciplines such as statistical and nuclear physics \cite{crisanti_products_1993}, population dynamics \cite{heyde1985confidence} and quantum mechanics \cite{bougerol_products_1985}. Their rigorous study began over sixty years ago, when Bellman \cite{bellman1954limit} studied the asymptotic behaviour of products of random matrices with strictly positive entries, corresponding to a weak law of large numbers. The seminal work of Furstenberg \& Kesten \cite{furstenberg_products_1960} and Furstenberg \cite{furstenberg_noncommuting_1963} strengthens this to a strong law for more general matrices.  Oseledec~\cite{Oseledec1968} extended this further to matrix cocycles of dynamical systems.

Here we consider the random product with~$N$ terms of the two matrices~$\{A_1,A_2\}$,
\begin{equation}
  \M_N = \prod_{k=1}^N A_{i_k},
  \qquad i_k \in \{1,2\},
  \label{eq:calMN}
\end{equation}
where the~$i_k$ are i.i.d.\ and the two index values have equal
probability~$1/2$.  It will often be convenient to write
\begin{equation}
  A = A_1 \quad\text{and}\quad B = A_2.
\end{equation}
The Lyapunov exponent is defined by
\begin{equation}
  \lambda = \lim_{N\rightarrow\infty}
  \frac{1}{N}\, \E\log\lVert\M_N\rVert
  \label{eq:lambdadef}
\end{equation}
where~$\lVert\cdot\rVert$ is some matrix norm.  The Furstenberg--Kesten
theorem~\cite{furstenberg_products_1960, furstenberg_noncommuting_1963} states
that the limit~\eqref{eq:lambdadef} exists, and is positive under fairly weak
assumptions on the~$A_i$, satisfied by the matrices we will be using. The Lyapunov exponent can be equivalently defined using a vector norm rather than a matrix norm as
\begin{equation}
  \lambda = \lim_{N\rightarrow\infty}
  \frac{1}{N}\, \E\log\lVert X_N\rVert ,
  \qquad
  X_N = \M_N X_0,
  \label{eq:lambda2def}
\end{equation}
for an arbitrary  vector $X_0$. In this paper we use the definition given by (\ref{eq:lambda2def}), and the choice of initial $X_0$ will be clear.

There is a paucity of exact results concerning Lyapunov exponents for random
matrices, as famously lamented by Kingman~\cite[p.~897]{kingman_subadditive_1973}. One well-known upper bound (described by \cite{protasov2013lower} as {\em ``the most popular upper bound in the literature''}) is easily derived from the submultiplicativity of $\lVert\cdot\rVert$.  For two matrices chosen with equal probability, let
\begin{equation}\label{eq:standard_bound}
E_k = \frac{1}{k}\, \E  \log \lVert C \rVert ,
\end{equation}
with $C \in \mathcal{A}^k$, where $\mathcal{A}^k$ is the set of all $2^k$ products of matrices of length $k$.  The numbers $E_k$ converge monotonically to $\lambda$ from above as $k \to \infty$ for any choice of matrix norm, although according to \cite{protasov2013lower} the Euclidean norm is usual.  In \cite{key1990lower} the bound is described as {\em ``easy, if not efficient''}, since the number of matrix product calculations required increases exponentially with~$k$.

Further progress in this direction has tended to be either for specific simple cases, or algorithmic procedures leading to (sometimes very accurate) approximations. For example, \cite{key1987computable} and \cite{pincus1985strong} discuss cases where the Lyapunov exponent can be computed exactly, in particular when matrices can be grouped in commuting blocks.  Chassain \etal \cite{Chassain1984} establish the distribution for the matrix product, in terms of a continued fraction, in the case that the matrices are $2 \times 2$ shear matrices, but observe that even for these simple matrices, the Lyapunov exponent is still unobtainable. A similar approach allowed Viswanath~\cite{viswanath_random_2000} to give a formula for the exponent in the case of matrices which give rise to a random Fibonacci sequence (this was extended by Janvresse~\etal\cite{janvresse_how_2007}).  An exact expression for $\lambda$ as the sum of a convergent series in the case for which one matrix is singular was given by Lima \& Rahibe \cite{lima_exact_1994}. Analytic expressions for $\lambda$ have also been obtained for large, sparse matrices \cite{cook_lyapunov_1990}, and for classes of $2 \times 2$ matrices in terms of explicitly expressed probability distributions \cite{mannion_products_1993, marklof_explicit_2008}. Pollicott
\cite{pollicott_maximal_2010} recently gave a cycle expansion formula that
allows a very accurate computation for a class of matrices.
  Protasov \& Jungers~\cite{protasov2013lower}  obtain an efficient algorithm for Lyapunov exponent bounds
using invariant cones for  matrices  with non-negative entries, concentrating on
generality and efficiency (they test their algorithm on examples up to dimension 60).

The difficulty of calculating Lyapunov exponents for products of random matrices can be seen in the variety of approaches. All the above results and algorithms, with the exception of those for random Fibonacci sequences, hold only for matrices with non-negative entries. Moreover, analytic expressions tend to be given in terms of probability distributions, continued fractions, or convergent series. In the present work we aim to provide rigorous and explicit upper and lower bounds for $\lambda$ for the same non-singular matrices as studied in \cite{Chassain1984}.

The matrices in question are shear matrices, which are of particular interest in several fluid mixing problems \cite{DAlessandro1999,stroock2002chaotic,khakhar1987case,aref1984stirring} and devices known as `taffy pullers' \cite{Boyland2000,Thiffeault2006,Finn2011}. The principle of mixing by chaotic advection can be summarised as {\em repeated stretching in transverse directions} \cite{ottino1989kinematics,sturman2006mathematical}. Many industrial mixing devices are designed on this basis, with the most fundamental model being periodic application of transverse shear matrices \cite{DAlessandro1999,stroock2002chaotic,khakhar1987case,aref1984stirring}.  Our work applies to devices where the mixing action is random.

For the problems of passive scalar decay and random taffy pullers, knowledge of the Lyapunov exponent is insufficient~\cite{Antonsen1996,Haynes2005,ThiffeaultAosta2004}.  We require more refined information via the growth rate of the $q$th moment of the matrix product norm. In particular, define the \emph{generalised Lyapunov exponents}~\cite{crisanti_generalized_1988, crisanti_products_1993} as
\begin{equation}\label{eq:gle}
  \ell_{\text{mat}}(q) = \lim_{N\rightarrow\infty}
  \frac{1}{N}\,\log\E \lVert \M_N \rVert^q.
\end{equation}
Again this can be restated using a vector norm:
\begin{equation}\label{eq:gle2}
  \ell(q) = \lim_{N\rightarrow\infty}
  \frac{1}{N}\,\log\E \lVert X_N \rVert^q,
\end{equation}
where~$X_N$ is as defined in~\eqref{eq:lambda2def}, but here we must observe that these definitions are not equivalent, particularly for~$q<0$.  In the present paper we will use the vector norm definition~\eqref{eq:gle2}.

\section{Rigorous bounds for Lyapunov exponents}

We derive rigorous and explicit bounds for Lyapunov exponents and generalised Lyapunov exponents by reformulating the problem, grouping the matrices together.  Assume without loss of generality that the first
matrix in the product~\eqref{eq:calMN} is~$A_1=A$.  By grouping~$A$'s and~$B$'s together into~$J$ blocks the random product~\eqref{eq:calMN} can be
written
\begin{equation}
  \M_{N_J} = \prod_{j=1}^J A^{a_j} B^{b_j},
  \qquad
  a_j+b_j = n_j,
  \qquad
  \sum_{j=1}^J n_j = N_J,
  \label{eq:calMdef}
\end{equation}
with~$1 \le a_i,b_i < n_i$, so~$n_i \ge 2$.  Now it is the~$a_i$ and~$b_i$
that are the i.i.d.\ random variables, with identical probability
distribution~$P(x) = 2^{-x}$, $x\ge1$.  Hence the length of each block is governed by the joint
distribution~$P(a,b) = P(a)\,P(b) = 2^{-(a+b)}$.  We have the expected
values~$\E a = \E b = 2$, so $\E n = 4$.

Let us now take the specific matrices
\begin{equation}
  A = \begin{pmatrix} 1 & 0 \\ \alpha & 1 \end{pmatrix},\qquad
  B = \begin{pmatrix} 1 & \beta \\ 0 & 1 \end{pmatrix}, \qquad
 \K_{ab} \ldef A^a B^b = \begin{pmatrix} 1 & b\beta \\ a\alpha & 1+a\alpha b\beta \end{pmatrix}.
  \label{eq:ABalphabeta}
\end{equation}
We consider first the case $\alpha, \beta >0$, for which $\K_{ab}$ is positive definite, and hyperbolic $\forall a, b \ge 1$. Although our technique holds for all positive $\alpha, \beta$, we state our results for $\alpha, \beta \ge 1$. This is partly due to ease of exposition, but also because in many applications $\alpha$ and $\beta$ would be assumed to be integers, so that a map induced by $\K_{ab}$ is continuous on the 2-torus. In particular, the algebraically simplest case $\alpha = \beta = 1$ corresponds to the generators of the 3-braid group seen in many studies of topological mixing \cite{Boyland2000,Thiffeault2006,Finn2011}. We then allow negative entries; in particular we consider $\alpha <0< \beta$ (note that $\alpha>0>\beta$ is essentially similar, while $\alpha<0, \beta<0$ is no different from the positive $\alpha, \beta$ case).  Now hyperbolicity is only guaranteed when the product $|\alpha \beta |$ is sufficiently large, and we require this property to obtain our results. We gain different bounds by considering different vector norms, a valid approach since the limit in (\ref{eq:lambda2def})  is independent of the choice of norm. In particular, we will consider the $L_1$, $L_2$ and $L_{\infty}$ norms. Which norm produces the most accurate bound depends on $\alpha$ and $\beta$. This is easily discerned by computation.

\subsection{Lyapunov exponents}

Our theorems are stated in terms of infinite sums of products of an exponentially decreasing term and a choice of (logarithm of)  increasing algebraic function, and so all obviously converge.

\begin{theorem}\label{thm:global_cone_bounds}
The Lyapunov exponent $\lambda(\alpha, \beta)$ for the product $\M_N$ for $\alpha, \beta \ge 1$  satisfies
$$
\max_{k \in \{ 1, 2, \infty\}} \mathcal{L}_k (\alpha, \beta) \le 4\lambda (\alpha, \beta) \le \min_{k \in \{ 1, 2, \infty\}} \mathcal{U}_k (\alpha, \beta)
$$
where
\begin{eqnarray*}
\mathcal{L}_k(\alpha,\beta)&=& \sum_{a,b=1}^\infty2^{-a-b}\log \phi_k (a, b, \alpha, \beta)\\
\mathcal{U}_k(\alpha,\beta)&=& \sum_{a,b=1}^\infty2^{-a-b}\log \psi_k (a, b, \alpha, \beta),
\end{eqnarray*}
and
\begin{eqnarray*}
\phi_1 (a, b, \alpha, \beta) &=& 1+\tfrac{\alpha}{1+\alpha} \left( a+ b\beta + a\alpha b\beta \right)  \\
\phi_2 (a, b, \alpha, \beta) &=& \min
 \begin{cases} ((1+a\alpha b\beta)^2 + b^2\beta^2)^{1/2} & \\
   \left( \tfrac1{1+\alpha^2}\left(\alpha^2(1+a+a\alpha b\beta)^2 + (1+\alpha b\beta)^2\right) \right)^{1/2}  &  \\
\end{cases}\\
\phi_{\infty} (a, b, \alpha, \beta) &=& 1+a\alpha b\beta   \\
\psi_1 (a, b, \alpha, \beta) &=& 1+b\beta+a\alpha b\beta  \\
\psi_2 (a, b, \alpha, \beta) &=&   \left( \tfrac{1}{2} \left( 2 + \mathcal{C}_{a\alpha b\beta}
  + \sqrt{\mathcal{C}_{a\alpha b\beta} (\mathcal{C}_{a\alpha b\beta} + 4)} \right)\right)^{1/2} \\
\psi_{\infty} (a, b, \alpha, \beta) &=&   1+a+a\alpha b\beta
\end{eqnarray*}
where $\mathcal{C}_{a\alpha b\beta} = (a\alpha+b\beta)^2 + (a\alpha b \beta)^2$.
\end{theorem}

Losing a little sharpness, the $L_{\infty}$-norm bounds provide a pair of simpler expressions with no infinite sums, stated in:
\begin{corollary}\label{cor:explicit_bounds}
The Lyapunov exponent $\lambda (\alpha, \beta)$ for the product $\M_N$ for $\alpha, \beta \ge 1$ satisfies
$$
\kappa + \log \alpha \beta \le 4\lambda \le \kappa + \log(\sqrt{\alpha \beta} + 1/\sqrt{\alpha \beta}) + \tfrac{1}{2} \log (1+\alpha \beta),
$$
where
$$
\kappa = \sum_{a,b=1}^{\infty} 2^{-a-b} \log ab \approx 1.0157\ldots .
$$
\end{corollary}

Theorem \ref{thm:global_cone_bounds} is obtained by considering a cone in tangent space which is invariant for all $a$ and $b$. We can improve these estimates by recognising that a smaller cone can be used for certain iterates of the map. In particular, we use the fact that since $a_i$ and $b_i$ are independent geometric distributions, $P(a=b) = P(a>b) = P(b>a)$ to give

\begin{theorem}\label{thm:improved_cone_bounds}
The Lyapunov exponent $\lambda(\alpha, \beta)$ for the product $\M_N$ for $\alpha, \beta \ge 1$ satisfies
$$
\max_{k \in \{ 1, 2, \infty\}} \hat{\mathcal{L}}_k (\alpha, \beta) \le 4\lambda (\alpha, \beta) \le \min_{k \in \{ 1, 2, \infty\}} \hat{\mathcal{U}}_k (\alpha, \beta)
$$
where
\begin{eqnarray*}
\hat{\mathcal{L}}_k(\alpha,\beta)&=& \sum_{a,b=1}^\infty2^{-a-b}\log \left( \frac{1}{3} \sum_{m=1}^3 \hat{\phi}_k^{(m)} (a, b, \alpha, \beta)\right)\\
\hat{\mathcal{U}}_k(\alpha,\beta)&=& \sum_{a,b=1}^\infty2^{-a-b}\log \left( \frac{1}{3} \sum_{m=1}^3 \hat{\psi}_k^{(m)} (a, b, \alpha, \beta) \right),
\end{eqnarray*}
and
\begin{eqnarray*}
\hat{\phi}_1^{(1)} (a, b, \alpha, \beta)&=& \phi_1(a, b, \alpha, \beta) \\
\hat{\phi}_1^{(2)} (a, b, \alpha, \beta) &=&\frac{\alpha \left(\alpha \beta + 2\right) \left(a \alpha b \beta + b \beta + 1\right) + \left(a \alpha + 1\right) \left(\alpha \beta + 1\right)}{\alpha \left(\alpha \beta + \beta + 2\right) + 1} \\
\hat{\phi}_1^{(3)}(a, b, \alpha, \beta) &=& \frac{\alpha \left(2 \alpha \beta + 3\right) \left(a \alpha b \beta + b \beta + 1\right) + \left(a \alpha + 1\right) \left(\alpha \beta + 1\right)}{\alpha \left(2 \alpha \beta + \beta + 3\right) + 1} \\
\hat{\phi}_2^{(1)}(a, b, \alpha, \beta) &=& \phi_2(a, b, \alpha, \beta) \\
\hat{\phi}_2^{(2)}(a, b, \alpha, \beta) &=&  \min
\begin{cases}
 ((1+a\alpha b\beta)^2 + b^2\beta^2)^{1/2} & \\
\left( \frac{\left( (1+\alpha\beta + \alpha b\beta(2+\alpha\beta))^2 + (a\alpha(1+\alpha \beta) + \alpha (2+\alpha \beta)(1+a\alpha b\beta)^2 \right)}{(1+\alpha \beta)^2+\alpha^2(2+\alpha\beta)^2}  \right)^{1/2} &  \\
\end{cases}
\\
\hat{\phi}_2^{(3)}(a, b, \alpha, \beta) &=&  \min
\begin{cases}
 ((1+a\alpha b\beta)^2 + b^2\beta^2)^{1/2} & \\
\left( \frac{\left( (1+\alpha\beta + \alpha b\beta(3+2\alpha\beta))^2 + (a\alpha(1+\alpha \beta) + \alpha (3+2\alpha \beta)(1+a\alpha b\beta)^2 \right)}{(1+\alpha \beta)^2+\alpha^2(3+2\alpha\beta)^2}  \right)^{1/2} &  \\
\end{cases}\\
\hat{\phi}_{\infty}^{(m)} (a, b, \alpha, \beta)&=& \phi_{\infty}(a, b, \alpha, \beta) \mbox{ for $m = 1, 2, 3$}
\end{eqnarray*}
and
\begin{eqnarray*}
\hat{\psi}_{1}^{(m)} (a, b, \alpha, \beta)&=& \psi_{1}(a, b, \alpha, \beta) \mbox{ for $m = 1, 2, 3$} \\
\hat{\psi}_{2}^{(m)} (a, b, \alpha, \beta)&=& \psi_{2}(a, b, \alpha, \beta) \mbox{ for $m = 1, 2, 3$}\\
\hat{\psi}_{\infty}^{(1)}(a, b, \alpha, \beta) &=& \psi_{\infty}(a, b, \alpha, \beta) \\
\hat{\psi}_{\infty}^{(2)}(a, b, \alpha, \beta) &=&
 1+a\alpha b\beta + \frac{a(1+\alpha \beta)}{2+\alpha \beta} \\
\hat{\psi}_{\infty}^{(3)}(a, b, \alpha, \beta) &=&
 1+a\alpha b\beta + \frac{a(1+\alpha \beta)}{3+2\alpha \beta}\,.
\end{eqnarray*}

\end{theorem}

\subsection{Generalised Lyapunov exponents}

We can use the functions defined above to bound the generalised Lyapunov exponents for each $q$:

\begin{theorem}\label{thm:gle_bounds}
We have, for $\alpha, \beta \ge 1$,
\begin{eqnarray*}
4\ell(q, \alpha, \beta) &\ge& \begin{cases} \max_{k \in \{ 1, 2, \infty\} } \left\{ \log \sum_{a, b=1}^{\infty} 2^{-a-b} (\phi_k (a,b,\alpha, \beta))^q\right\}  & q\ge 0 \\
 \max_{k \in \{ 1, 2, \infty\} } \left\{ \log \sum_{a, b=1}^{\infty} 2^{-a-b}  (\psi_k (a,b,\alpha, \beta))^q\right\}  & q<0 \end{cases} \\
4\ell(q, \alpha, \beta) &\le& \begin{cases} \min_{k \in \{ 1, 2, \infty\} } \left\{ \log \sum_{a, b=1}^{\infty} 2^{-a-b} (\psi_k (a,b,\alpha, \beta))^q\right\}  & q\ge 0 \\
 \min_{k \in \{ 1, 2, \infty\} } \left\{ \log \sum_{a, b=1}^{\infty} 2^{-a-b}  (\phi_k (a,b,\alpha, \beta))^q\right\}  & q<0 \end{cases}
\end{eqnarray*}
and the more accurate  expressions
\begin{eqnarray*}
4\ell(q, \alpha, \beta) &\ge& \begin{cases} \max_{k \in \{ 1, 2, \infty\} } \left\{ \log \frac{1}{3} \sum_{a, b=1}^{\infty} 2^{-a-b} \sum_{m=1}^3  (\hat{\phi}_k^{(m)} (a,b,\alpha, \beta))^q\right\} & q\ge 0  \\
 \max_{k \in \{ 1, 2, \infty\} } \left\{ \log \frac{1}{3}\sum_{a, b=1}^{\infty} 2^{-a-b} \sum_{m=1}^3  (\hat{\psi}_k^{(m)} (a,b,\alpha, \beta))^q\right\} & q<0 \end{cases} \\
4\ell(q, \alpha, \beta) &\le& \begin{cases} \min_{k \in \{ 1, 2, \infty\} } \left\{ \log \frac{1}{3} \sum_{a, b=1}^{\infty} 2^{-a-b} \sum_{m=1}^3  (\hat{\psi}_k^{(m)} (a,b,\alpha, \beta))^q\right\} & q\ge 0  \\
 \min_{k \in \{ 1, 2, \infty\} } \left\{ \log \frac{1}{3}\sum_{a, b=1}^{\infty} 2^{-a-b} \sum_{m=1}^3  (\hat{\phi}_k^{(m)} (a,b,\alpha, \beta))^q\right\} & q<0 \end{cases}
\end{eqnarray*}
with $\phi_k, \psi_k, \hat{\phi}_k^{(m)}$ and $\hat{\psi}_k^{(m)}$ defined as above.
\end{theorem}

An immediate observation is that since all the functions $\phi, \psi, \hat{\phi}_k^{(m)}$ and $\hat{\psi}_k^{(m)}$ are greater than 1 for all $a,b \ge 1$, $\alpha, \beta \ge 0$, and since $\sum_{a,b=1}^{\infty} 2^{-a-b} = 1$, the bounds for $\ell(q,\alpha,\beta)$ grow from 0 for positive $q$ and decay from zero for negative $q$. This apparently contradicts Proposition 2 of \cite{vanneste2010estimating}, which states that there is always a minimum in the curve for $\ell(q)$, and  in particular states that $\ell(-2)=0$ if the 2-dimensional matrices in question have determinant 1. The existence of the invariant cone for these shear matrices guarantees that a vector is expanded at every application of $A$ or $B$, which forces $\ell(q)$ to be monotonic. In \cite{vanneste2010estimating}, the assumption is made that the linear operator corresponding to the generalised Lyapunov exponent has the same spectrum as its adjoint, a property precluded by the invariant cone. The fact  $\phi, \psi, \hat{\phi}_k^{(m)}$ and $\hat{\psi}_k^{(m)} \ge 1$ is also the reason that $\phi$ and $\psi$ exchange roles in upper and lower bounds for positive and negative $q$.

When~$q$ is a positive integer we can evaluate this easily by expanding the power~$q$.  Since
$$
\sum_{a, b=1}^\infty 2^{-a-b} a^nb^m = \left(\sum_{a=1}^{\infty} 2^{-a} a^n\right)\left(\sum_{b=1}^{\infty} 2^{-b} b^m\right)
$$
we require values of the polylogarithm $\Li_{-q}(\tfrac12)$, defined by
$$
  \Li_s(z) = \sum_{k=1}^\infty \frac{z^k}{k^s}\,.
$$
For integer $q = -s$ we have special values
$$
\sum_{a=1}^{\infty} 2^{-a} a^n = 1,2,6,26,150,1082,9366, \ldots \mbox{ for $n = 0,1,2,3,4,5,6,\ldots$}
$$
and so the $L_{\infty}$ norm, for example,  gives
\begin{corollary}\label{cor:gle_values}
Generalised Lyapunov exponents in the case $\alpha = \beta = 1$ are bounded by:
\begin{eqnarray*}
\tfrac{1}{4}\log 5 &\le& \ell(1,1,1) \le \tfrac{1}{4}\log 7 \\
\tfrac{1}{4}\log 45 &\le& \ell(2,1,1) \le \tfrac{1}{4}\log 79 \\
\tfrac{1}{4}\log 797 &\le& \ell(3,1,1) \le \tfrac{1}{4}\log 1543 \\
\tfrac{1}{4}\log 25437 &\le& \ell(4,1,1) \le \tfrac{1}{4}\log 50531 \\
\tfrac{1}{4}\log 1290365 &\le& \ell(5,1,1) \le \tfrac{1}{4}\log 2578567 .
\end{eqnarray*}
\end{corollary}

Theorem~\ref{thm:gle_bounds} also allows explicit estimates on topological entropy for the random matrix product, given by the generalised Lyapunov exponent with $q=1$.
\begin{corollary}\label{cor:top_entropy}
The topological entropy $\ell(1,\alpha,\beta)$ in the case $\alpha, \beta \ge 1$  is bounded by
$$
\log(1+4\alpha \beta) \le 4\ell(1,\alpha,\beta) \le \log (3+4\alpha\beta).
$$
\end{corollary}

\subsection{Matrices with negative entries}\label{sec:neg}

The case where the direction of one of the shears is reversed (that is, allowing negative entries in the matrix) can be tackled in an almost identical manner, with one important condition. Taking $\alpha<0<\beta$ (the case $\alpha>0>\beta$ is essentially identical), the matrix $\K_{11} = AB$ is hyperbolic only when the product $|\alpha \beta|  >4$. In the following, for simplicity, we will assume $\alpha < -2$, $\beta>2$ to achieve this\footnote{In fact, we might assume that $\alpha = -\beta$. Otherwise we change coordinates, as in \cite{przytycki1983ergodicity} to $(x, \sqrt{|\alpha/\beta|}y)$. But here we retain $\alpha \ne -\beta$ to show explicitly the dependence of the bounds on choosing unequal strengths of twists.}.

\begin{theorem}\label{thm:global_cone_bounds_rev}
The Lyapunov exponent $\lambda(\alpha, \beta)$ for the product $\M_N$ in the case $\alpha <-2, \beta>2$ satisfies
$$
\max_{k \in \{ 1, 2, \infty\}} \tilde{\mathcal{L}}_k (\alpha, \beta) \le 4\lambda (\alpha, \beta) \le \min_{k \in \{ 1, 2, \infty\}} \tilde{\mathcal{U}}_k (\alpha, \beta)
$$
where
\begin{eqnarray*}
\tilde{\mathcal{L}}_k(\alpha,\beta)&=& \sum_{a,b=1}^\infty2^{-a-b}\log \tilde{\phi}_k (a, b, \alpha, \beta)\\
\tilde{\mathcal{U}}_k(\alpha,\beta)&=& \sum_{a,b=1}^\infty2^{-a-b}\log \tilde{\psi}_k (a, b, \alpha, \beta),
\end{eqnarray*}
and
\begin{eqnarray*}
\tilde{\phi}_1 (a, b, \alpha, \beta) &=& \tfrac{1}{1-\Gamma} \left( b\beta + \Gamma - a\alpha b\beta - 1- a\alpha \Gamma\right)  \\
\tilde{\phi}_2 (a, b, \alpha, \beta) &=&   \left( \tfrac1{1+\Gamma^2}\left( (\Gamma+b\beta)^2 + (1+a\alpha \Gamma + a\alpha b\beta)^2  \right) \right)^{1/2} \\
\tilde{\phi}_{\infty} (a, b, \alpha, \beta) &=& -a\alpha b\beta - a\alpha \Gamma -1  \\
\tilde{\psi}_1 (a, b, \alpha, \beta) &=& b\beta-a\alpha b\beta-1  \\
\tilde{\psi}_2 (a, b, \alpha, \beta) &=&  ((-a\alpha b\beta-1)^2 + b^2\beta^2)^{1/2} \\
\tilde{\psi}_{\infty} (a, b, \alpha, \beta) &=& -a\alpha b\beta -1
\end{eqnarray*}
where
$$
\Gamma = - \frac{\beta}{2} + \sqrt{\left( \frac{\beta}{2}  \right)^2 + \frac{\beta}{\alpha}}.
$$
\end{theorem}

Again we can straightforwardly improve on the lower bounds by considering separately the cases when either, or both, of  $a$ and $b$ are equal to 1.

\begin{theorem}\label{thm:improved_cone_bounds_rev}
The Lyapunov exponent $\lambda(\alpha, \beta)$ for the product $\M_N$ in the case $\alpha <-2, \beta>2$ satisfies
$$
 \max_{k \in \{ 1, 2, \infty\}} \hat{\tilde{\mathcal{L}}}_k (\alpha, \beta) \le 4\lambda (\alpha, \beta) \le \hat{\tilde{\mathcal{U}}}_{\infty} (\alpha, \beta)
$$
where
\begin{eqnarray*}
\hat{\tilde{\mathcal{L}}}_k(\alpha,\beta)&=& \sum_{a,b=1}^\infty2^{-a-b}\log \frac{1}{4} \sum_{m_a, m_b=1}^{2} \hat{\tilde{\phi}}_k^{(m_a,m_b)} (a, b, \alpha, \beta)\\
\hat{\tilde{\mathcal{U}}}_{\infty} (\alpha,\beta)&=& \sum_{a,b=1}^\infty2^{-a-b}\log \frac{1}{4} \sum_{m=1}^{4} \hat{\tilde{\psi}}_{\infty}^{(m)} (a, b, \alpha, \beta)
\end{eqnarray*}
and
\begin{eqnarray*}
\hat{\tilde{\phi}}_1^{(m_a,m_b)} (a, b, \alpha, \beta) &=& \tfrac{1}{1-\Gamma_{m_a,m_b}} \left( b\beta + \Gamma_{m_a,m_b} - a\alpha b\beta - 1- a\alpha \Gamma_{m_a,m_b}\right)  \\
\hat{\tilde{\phi}}_2^{(m_a,m_b)} (a, b, \alpha, \beta) &=& \left( \tfrac1{1+\Gamma_{m_a,m_b}^2}\left( (\Gamma_{m_a,m_b}+b\beta)^2 + (1+a\alpha \Gamma_{m_a,m_b} + a\alpha b\beta)^2  \right) \right)^{1/2}  \\
\hat{\tilde{\phi}}_{\infty}^{(m_a,m_b)} (a, b, \alpha, \beta) &=& -a\alpha b\beta - a\alpha \Gamma_{m_a,m_b} -1\\
\hat{\tilde{\psi}}_{\infty}^{(1)}  (a,b,\alpha,\beta)&=&  -a\alpha b\beta -1  -a\alpha\beta/(1+\alpha \beta))  \\
\hat{\tilde{\psi}}_{\infty}^{(2)}  (a,b,\alpha,\beta)&=&  -a\alpha b\beta -1 -a  \\
\hat{\tilde{\psi}}_{\infty}^{(3)}  (a,b,\alpha,\beta)&=&  \hat{\tilde{\psi}}_{\infty}^{(4)} (a,b,\alpha,\beta) =  \tilde{\psi}_{\infty} (a, b, \alpha, \beta)
\end{eqnarray*}
with
\begin{equation}
\Gamma_{m_a,m_b} = \frac{\Gamma+m_b\beta}{m_a\alpha \Gamma + m_a m_b \alpha \beta +1}
\end{equation}
for $m_a, m_b = 1,2$. Note that $\Gamma_{1,1} = L$.

\end{theorem}
We could also write improved upper estimates using $L_1$ and $L_2$ norms, but since these produce worse bounds than the $L_{\infty}$ norm in all cases we study here, we do not give these explicitly.

As before we can use the same estimates to give explicit bounds for generalised Lyapunov exponents.

\begin{theorem}\label{thm:gle_bounds_rev}
We have, for $\alpha <-2$ and $\beta>2$,
\begin{eqnarray*}
\fl\qquad 4\ell(q, \alpha, \beta) &\le& \begin{cases} \min_{k \in \{ 1, 2, \infty\} } \left\{ \log \sum_{a, b=1}^{\infty} 2^{-a-b} (\tilde{\phi}_k (a,b,\alpha, \beta))^q\right\}  & q\ge 0 \\
 \min_{k \in \{ 1, 2, \infty\} } \left\{ \log \sum_{a, b=1}^{\infty} 2^{-a-b}  (\tilde{\psi}_k (a,b,\alpha, \beta))^q\right\}  & q<0 \end{cases} \\
\fl\qquad 4\ell(q, \alpha, \beta) &\ge& \begin{cases} \max_{k \in \{ 1, 2, \infty\} } \left\{ \log \sum_{a, b=1}^{\infty} 2^{-a-b} (\tilde{\psi}_k (a,b,\alpha, \beta))^q\right\}  & q\ge 0 \\
 \max_{k \in \{ 1, 2, \infty\} } \left\{ \log \sum_{a, b=1}^{\infty} 2^{-a-b}  (\tilde{\phi}_k (a,b,\alpha, \beta))^q\right\}  & q<0 \end{cases}
\end{eqnarray*}
and the more accurate, but more complicated expressions
\begin{eqnarray*}
\fl\qquad 4\ell(q, \alpha, \beta) &\le& \begin{cases} \min_{k \in \{ 1, 2, \infty\} } \left\{ \log \frac{1}{4} \sum_{a, b=1}^{\infty} 2^{-a-b} \sum_{m_a=1}^2 \sum_{m_b=1}^2  (\hat{\tilde{\phi}}_k^{(m_a,m_b)} (a,b,\alpha, \beta))^q\right\} & q\ge 0  \\
 \min_{k \in \{ 1, 2, \infty\} } \left\{ \log \frac{1}{4}\sum_{a, b=1}^{\infty} 2^{-a-b} \sum_{m=1}^4  (\hat{\tilde{\psi}}_k^{(m)} (a,b,\alpha, \beta))^q\right\} & q<0 \end{cases} \\
\fl\qquad 4\ell(q, \alpha, \beta) &\ge& \begin{cases} \max_{k \in \{ 1, 2, \infty\} } \left\{ \log \frac{1}{4} \sum_{a, b=1}^{\infty} 2^{-a-b} \sum_{m=1}^4  (\hat{\tilde{\psi}}_k^{(m)} (a,b,\alpha, \beta))^q\right\} & q\ge 0  \\
 \max_{k \in \{ 1, 2, \infty\} } \left\{ \log \frac{1}{4}\sum_{a, b=1}^{\infty} 2^{-a-b} \sum_{m_a=1}^2 \sum_{m_b=1}^2 (\hat{\tilde{\phi}}_k^{(m_a,m_b)} (a,b,\alpha, \beta))^q\right\} & q<0 \end{cases}
\end{eqnarray*}
with $\tilde{\phi_k}, \tilde{\psi_k}, \hat{\tilde{\phi}}_k^{(m_a,m_b)}$ and $\hat{\tilde{\psi}}_k^{(m)}$ defined as above.
\end{theorem}

\section{Accuracy of the bounds}

\subsection{Lyapunov exponents}

For $\alpha=\beta=1$ we have bounds on the Lyapunov exponent given in table~\ref{tab:bounds}. The lowest upper bound ($\mathcal{U}_2$) and largest lower bound ($\hat{\mathcal{L}}_1$) differ by about $2.5\%$. The true value (from explicit calculation via a standard algorithm \cite{parker2012practical}) is 0.39625..., so the upper bound here is rather tighter than the lower.
\begin{table}
\centering
\begin{tabular}{|c|c|c|c|}
\hline
\multicolumn{1}{|c|}{}  & \multicolumn{2}{c|} {Invariant cone} &  \multicolumn{1}{c|} {Smaller cones} \\
\hline
Norm & Lower bound&Upper bound& Improved bounds \\
\hline
$L_1$ & $\mathcal{L}_1 (1,1)$ = 0.36886 & $\mathcal{U}_1(1,1)$ = 0.43835 & $\hat{\mathcal{L}}_1 (1,1)$ = 0.38561 \\
$L_2$ & $\mathcal{L}_2 (1,1)$ = 0.36347 & $\mathcal{U}_2(1,1)$ = 0.40277 & $\hat{\mathcal{L}}_2 (1,1)$ = 0.36864 \\
$L_{\infty}$ & $\mathcal{L}_{\infty} (1,1)$ = 0.34613 & $\mathcal{U}_{\infty}(1,1)$ = 0.43835 & $\hat{\mathcal{U}}_{\infty}(1,1)$ = 0.41350  \\
\hline
\end{tabular}
\caption{Bounds to five significant figures for the maximal Lyapunov exponent for the matrix product \eqref{eq:calMN} in the case $\alpha = \beta = 1$, where the true value (from numerical computation) is 0.39625... . }\label{tab:bounds}
\end{table}

Figure \ref{fig:le_figs} shows the accuracy of each bound from Theorem~\ref{thm:global_cone_bounds} increasing as $\alpha$ increases, with $\alpha = \beta \in [1,10]$. In figure~\ref{fig:le_bounds1} the bounds are all so close to the true value of $\lambda$ that the details of the graph are difficult to resolve. It is clear however, that the standard bound given by \eqref{eq:standard_bound} (plotted in cyan) is a worse upper bound than all others in the figure, despite being calculated from the expected value of matrix products of length $2^{12}$, and decreases in accuracy for this fixed $k$ for increasing $\alpha$.

Figure \ref{fig:le_error} shows the difference between the bounds and the true (numerically calculated) Lyapunov exponent. In this and other figures we colour bounds originating from $L_1$-, $L_2$- and $L_{\infty}$-norms green, red and blue respectively.  It shows that for increasing $\alpha$, upper bounds (solid lines) appear tighter than lower bounds (dashed lines). In black is shown the upper and lower bounds given in Corollary \ref{cor:explicit_bounds}, which are typically worse than those of Theorem~\ref{thm:global_cone_bounds}, but have the advantage of being explicit, finite expressions rather than infinite sums.

Figure \ref{fig:le_bounds_env} shows the envelope formed from the difference between upper and lower bounds derived from each norm, which does not require the explicit numerical calculation of the Lyapunov exponent to compute. To this figure we add, in figure~\ref{fig:le_better_bounds_env}, the corresponding bounds from Theorem~\ref{thm:improved_cone_bounds} which improve on Theorem~\ref{thm:global_cone_bounds} by considering the expected relation between the random variables $a_i$ and $b_i$. In black is the envelope formed from taking the minimum upper bound, and maximum lower bound for each value of $\alpha$. This improves on all bounds produced from a single norm.

\begin{figure}
\captionsetup[subfigure]{justification=raggedright}
\centering
\begin{subfigure}[t]{0.49\linewidth}
\centering
 \includegraphics[width=\linewidth]{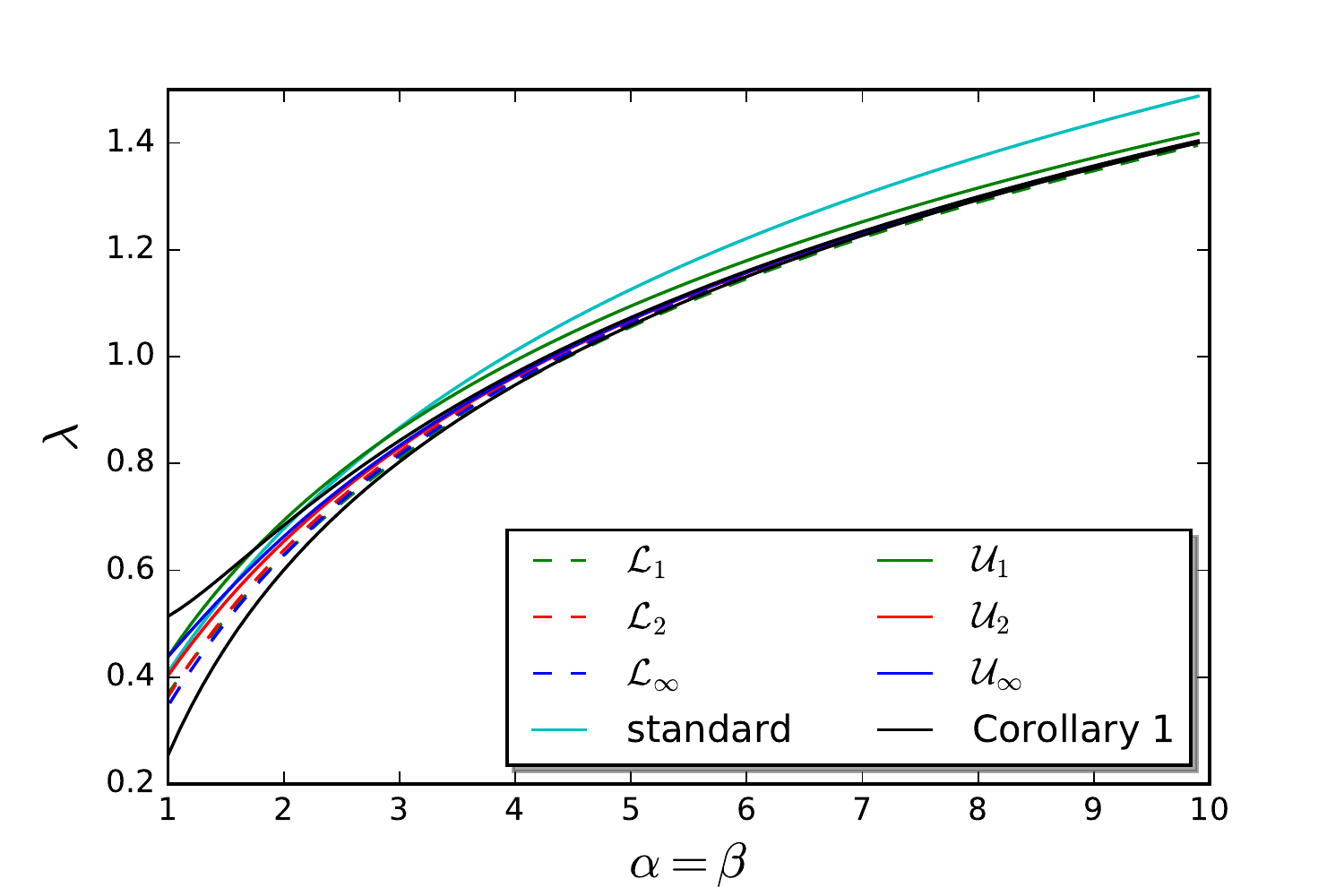}
 \caption{Numerical estimate of the Lyapunov exponent obtained by random
    multiplication of the matrices~\eqref{eq:ABalphabeta}, with bounds as given in Theorem~\ref{thm:global_cone_bounds}. Only the standard bound is appreciably far from the true value.}\label{fig:le_bounds1}
\end{subfigure}
\begin{subfigure}[t]{0.49\linewidth}
\centering
 \includegraphics[width=\textwidth]{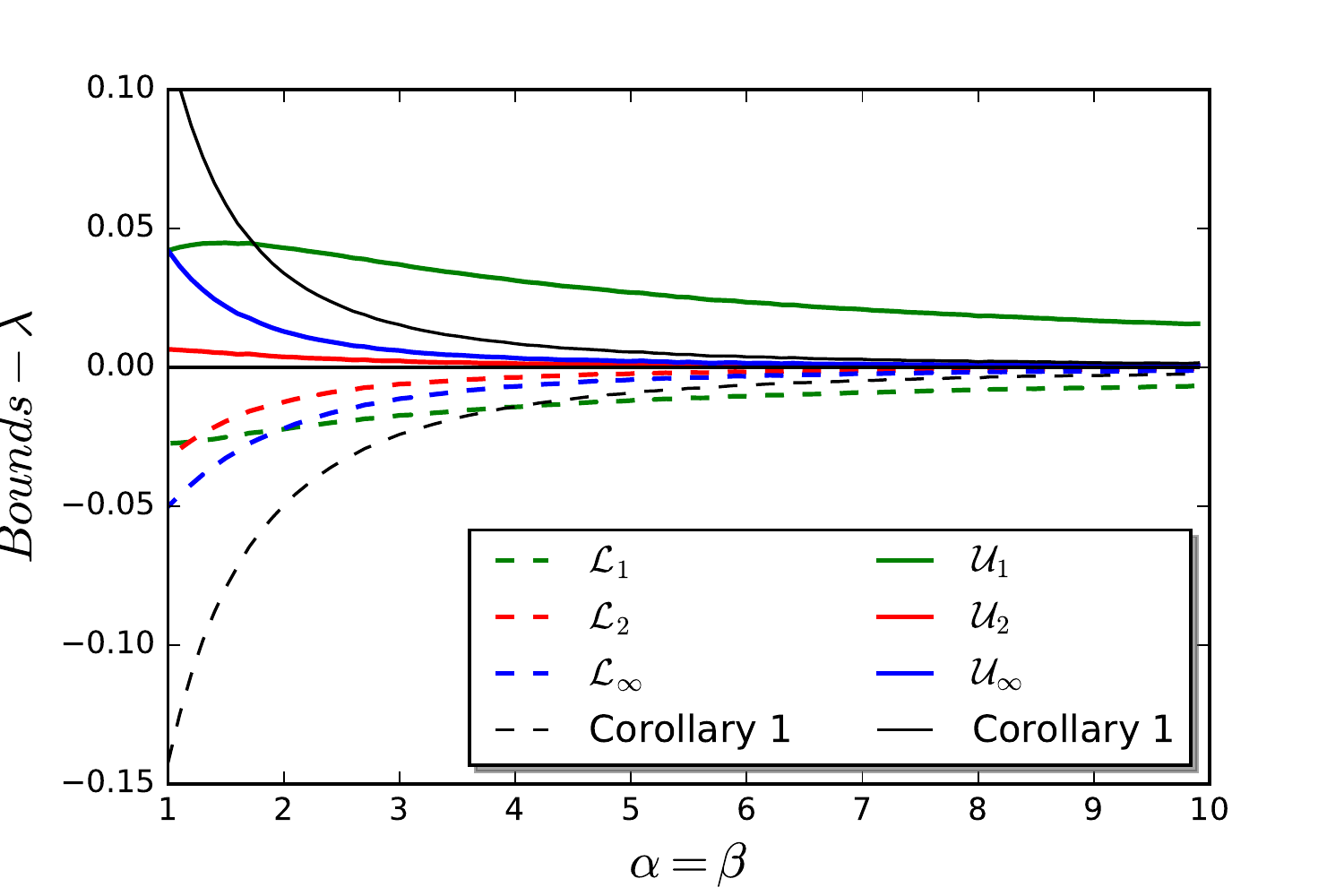}
 \caption{Errors in bounds from numerically calculated value. }\label{fig:le_error}
\end{subfigure}

\begin{subfigure}[t]{0.49\linewidth}
\centering
 \includegraphics[width=\linewidth]{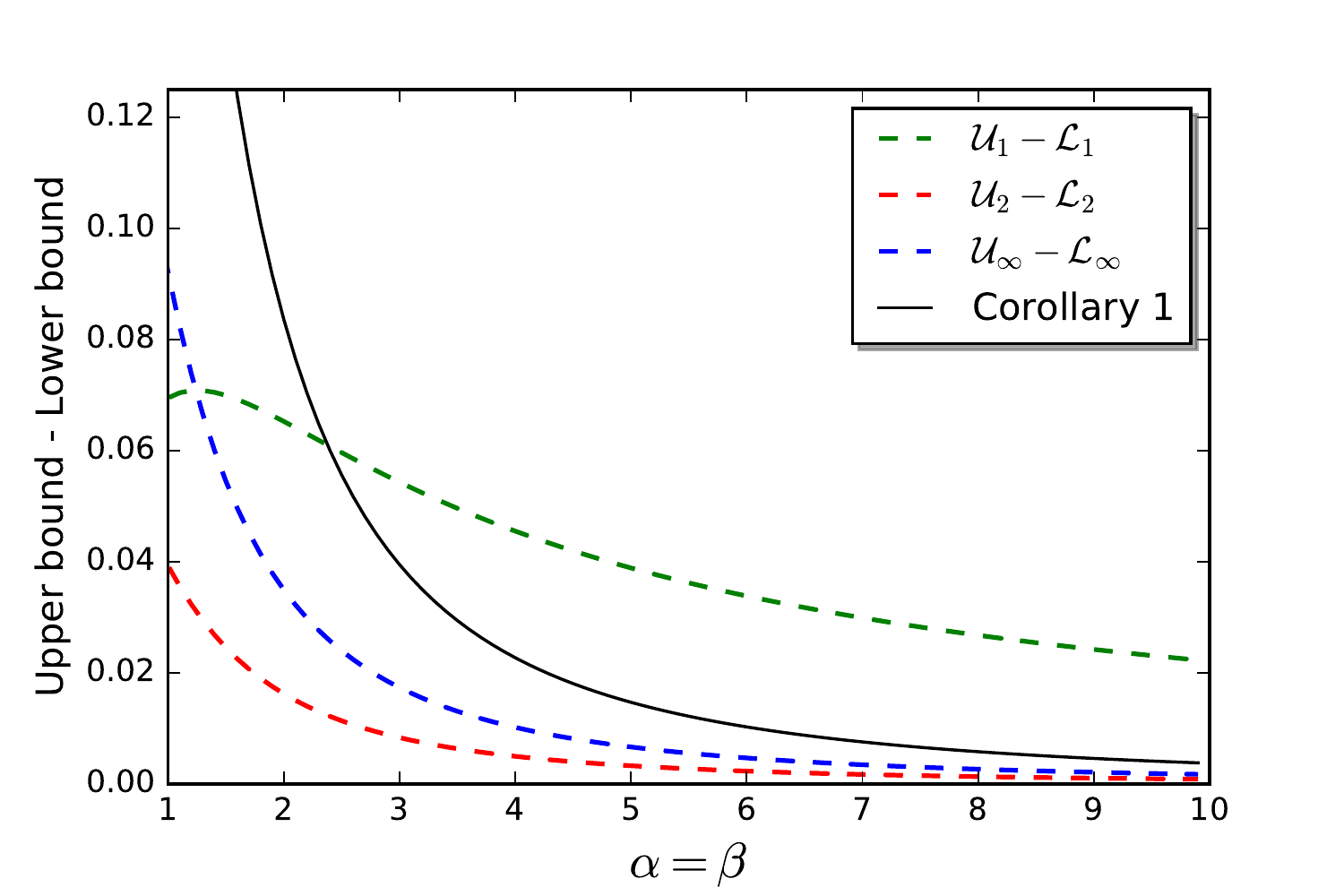}
 \caption{Difference between upper and lower bounds.}\label{fig:le_bounds_env}
\end{subfigure}
\begin{subfigure}[t]{0.49\linewidth}
\centering
 \includegraphics[width=\textwidth]{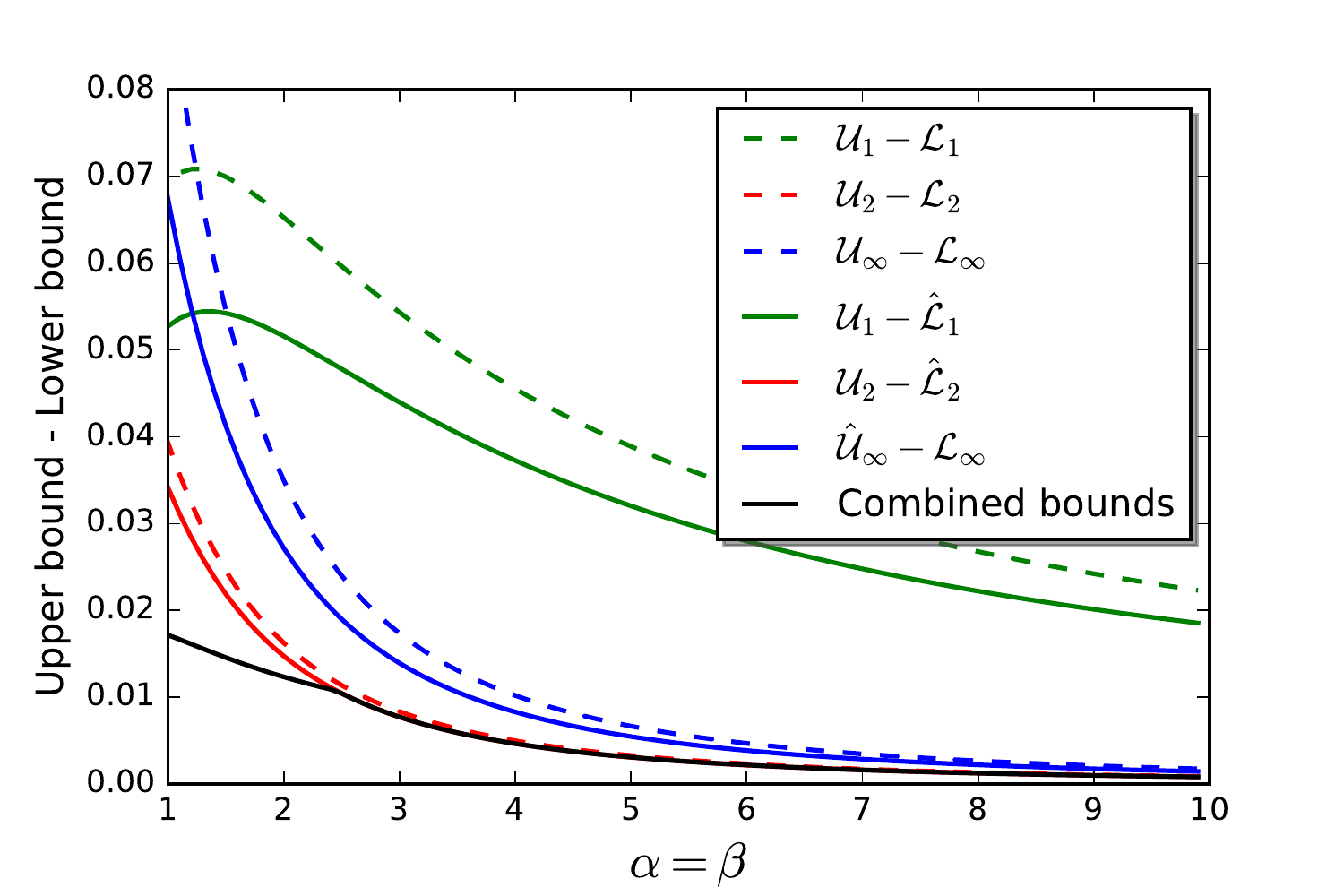}
 \caption{Envelope of bounds when improved cone is included.}\label{fig:le_better_bounds_env}
\end{subfigure}
\caption{Four different views of the accuracy of the upper and lower bounds for the Lyapunov exponent for the random matrix product (\ref{eq:calMN}) for $\alpha = \beta \in [1,10]$. In each case, bounds obtained from different norms are shown, in particular $L_1$ (green), $L_2$ (red) and $L_{\infty}$ (blue) are shown, with bounds from the global cone shown dashed, and the improved cone shown as a solid line. When shown, the standard bound is cyan.}\label{fig:le_figs}
\end{figure}

The increase of accuracy of the bounds with increasing strength of shear can be understood geometrically, as the size of the cone narrows with increasing shear, and dynamically, as the approach to the unstable eigenvalue is more rapid for matrices whose largest eigenvalue is greater. In figures \ref{fig:le_figs} and \ref{fig:rev_figs} it is clear that the upper and lower bounds approach the same curve for large $\alpha = \beta$. A simple calculation (using the $L_{\infty}$-norm) gives
\begin{eqnarray*}
4\lambda &\ge& \sum_{a,b=1}^{\infty} 2^{-a-b} \log (1+a\alpha b \beta) \\
&\ge& \sum_{a,b=1}^{\infty} 2^{-a-b} \log (a\alpha b \beta) \\
& = & \sum_{a,b=1}^{\infty} 2^{-a-b} \log (ab) + \sum_{a,b=1}^{\infty} 2^{-a-b} \log \alpha \beta \\
&\ge& \kappa + \log \alpha \beta.
\end{eqnarray*}
Meanwhile, for large $\alpha, \beta$ we also have
\begin{eqnarray*}
4\lambda &\le& \sum_{a,b=1}^{\infty} 2^{-a-b} \log (1+a+a\alpha b  \beta ) \\
&\sim& \sum_{a,b=1}^{\infty} 2^{-a-b} \log (a\alpha b \beta ) \\
& \sim &  \kappa + \log \alpha \beta ,
\end{eqnarray*}
and this indeed appears to be the asymptotic limit in the graphs shown  for $\alpha = \beta \to \infty$.

\subsection{Generalised Lyapunov exponents}

In figure \ref{fig:gle_figs} we show bounds for generalised Lyapunov exponents for $\alpha = \beta = 1$. Equivalent figures are increasingly accurate with increasing $\alpha, \beta$. Figure \ref{fig:gle_bounds} confirms that for this choice of matrices we do not have $\ell (-2) = 0$, and that there is no minimum in the curve of $\ell(q)$. Green, red and blue lines again show bounds originating from $L_1$-, $L_2$- and $L_{\infty}$-norms respectively, with the explicit expressions from corollary \ref{cor:gle_values} shown as black circles. Figure \ref{fig:gle_env} shows the envelope of the difference between upper and lower bounds, for each norm, and, in black, the envelope of best combined bounds.

\begin{figure}
\captionsetup[subfigure]{justification=raggedright}
\centering
\begin{subfigure}[t]{0.49\linewidth}
\centering
 \includegraphics[width=\linewidth]{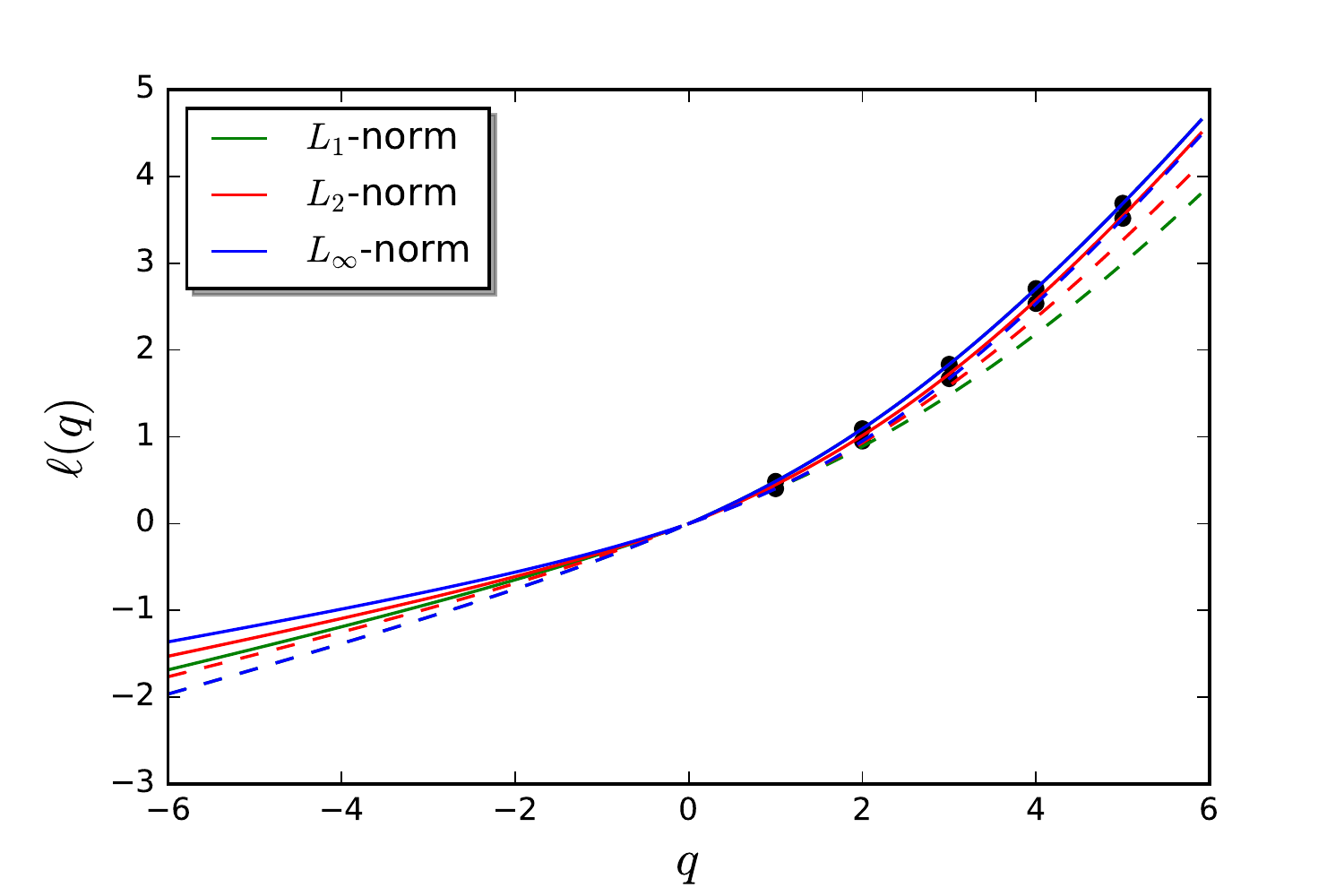}
 \caption{The bounds confirm that the curve of generalised Lyapunov exponents has no minimum at $q=-2$. }\label{fig:gle_bounds}
\end{subfigure}
\begin{subfigure}[t]{0.49\linewidth}
\centering
 \includegraphics[width=\textwidth]{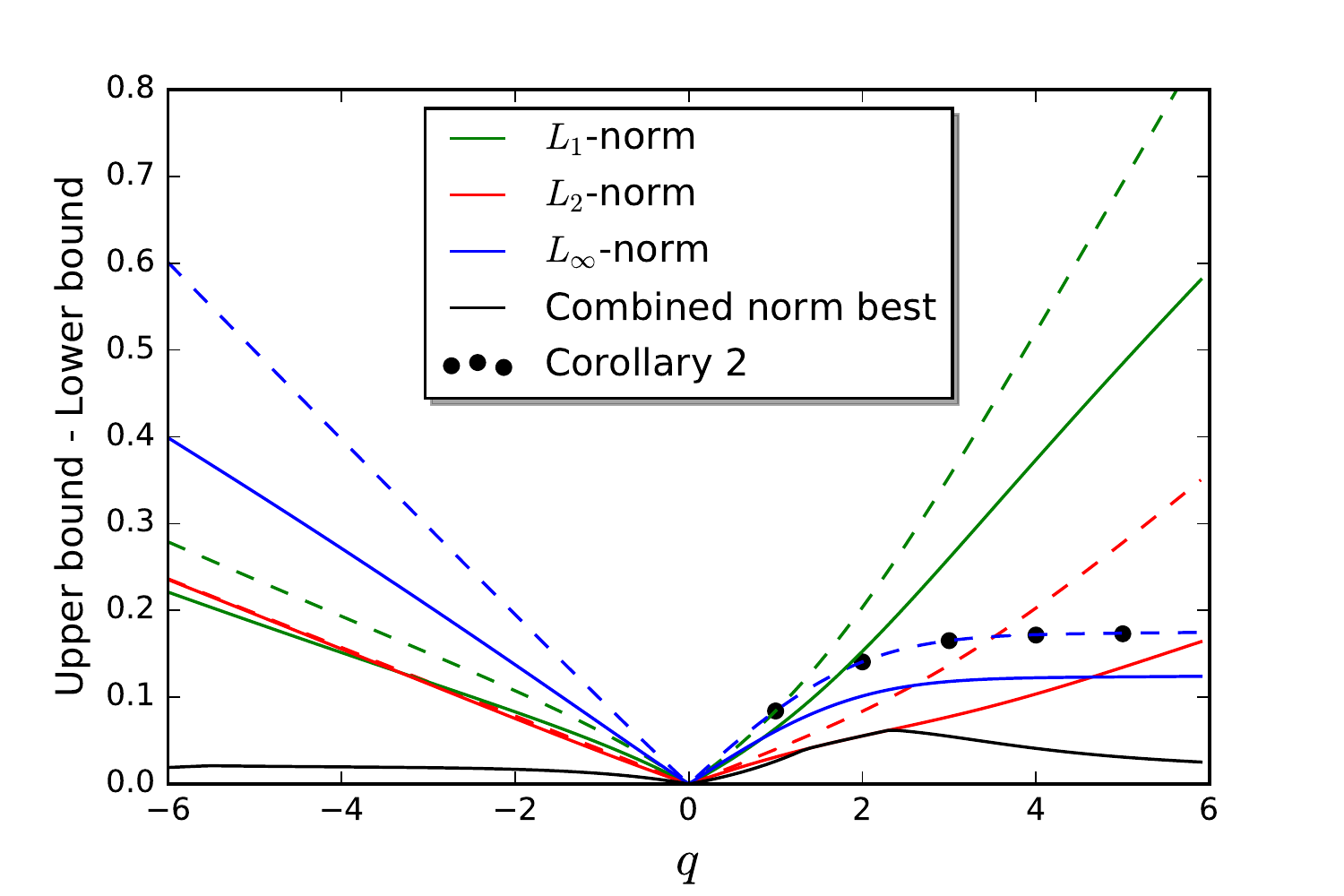}
 \caption{Difference between upper and lower bounds. Dashed lines represent bounds from Theorem~\ref{thm:global_cone_bounds}, while solid lines are those from~\ref{thm:improved_cone_bounds}. The black line represents the minimal envelope over all norms, and thus our best bounds.}\label{fig:gle_env}
\end{subfigure}
\caption{Bounds for the generalised Lyapunov exponent for the matrix product \ref{eq:ABalphabeta} with $\alpha = \beta =1$. As before, estimates originating from $L_1$-, $L_2$- and $L_{\infty}$-norms are given in green, red and blue respectively. For integer values of $q>0$, values from Corollary~\ref{cor:gle_values} are given as black circles. }\label{fig:gle_figs}
\end{figure}

\subsection{Negative shears}

Figure \ref{fig:rev_figs} shows the bounds for the case $\alpha <-2$, $\beta>2$. In these figures we set $\alpha = -\beta$, and observe that again, the increasing hyperbolicity from increasing $|\alpha|$ results in improving bounds. In this case the $L_{\infty}$-norm always gives the minimal envelope, as seen in figure \ref{fig:reverse_env}. Generalised Lyapunov exponents for $\alpha = -3$, $\beta =3$ are shown in figure \ref{fig:gle_figs_rev}.

\begin{figure}
\captionsetup[subfigure]{justification=raggedright}
\centering
\begin{subfigure}[t]{0.49\linewidth}
\centering
 \includegraphics[width=\linewidth]{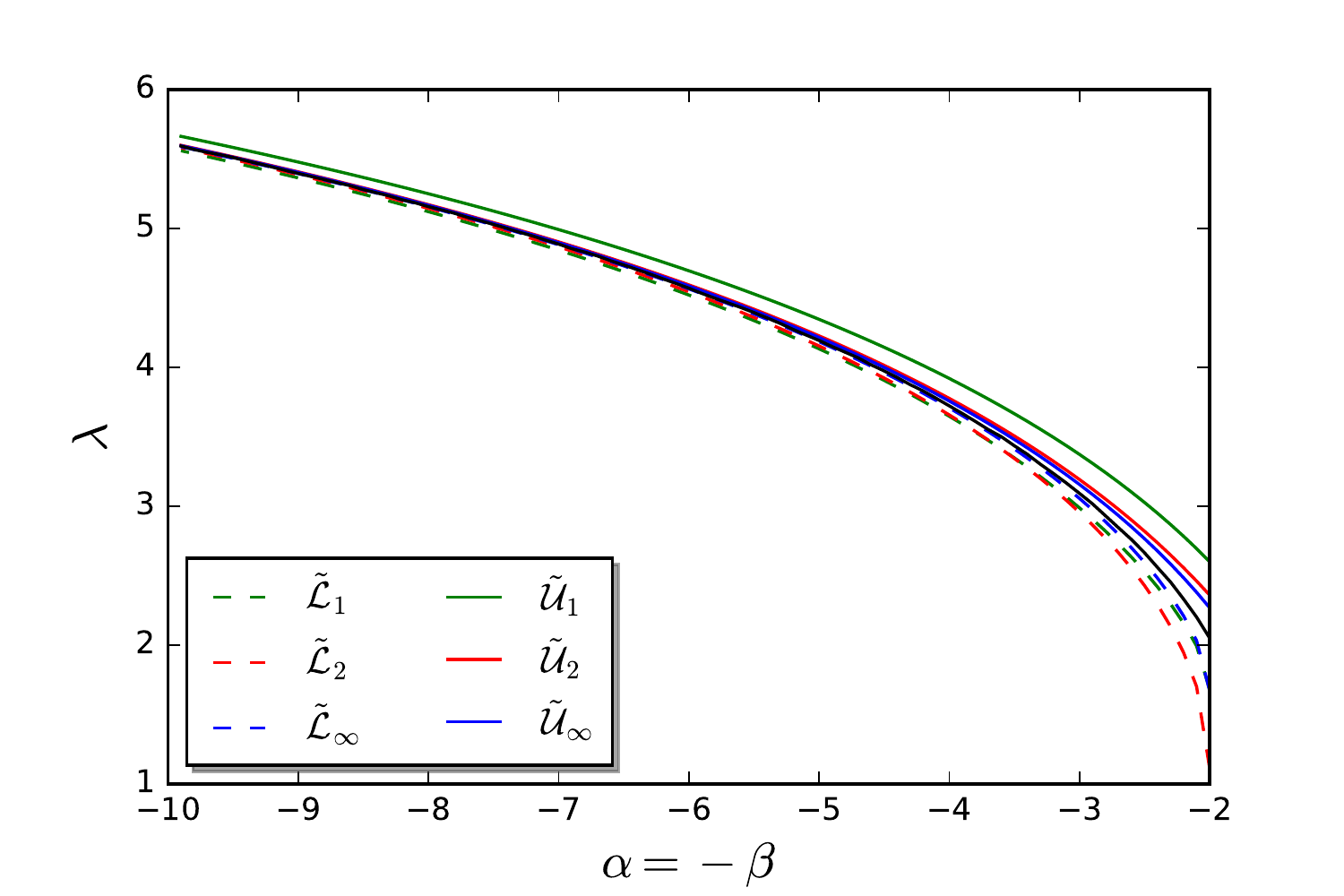}
 \caption{Bounds from Theorem~\ref{thm:global_cone_bounds_rev}.}\label{fig:reverse_bounds}
\end{subfigure}
\begin{subfigure}[t]{0.49\linewidth}
\centering
 \includegraphics[width=\textwidth]{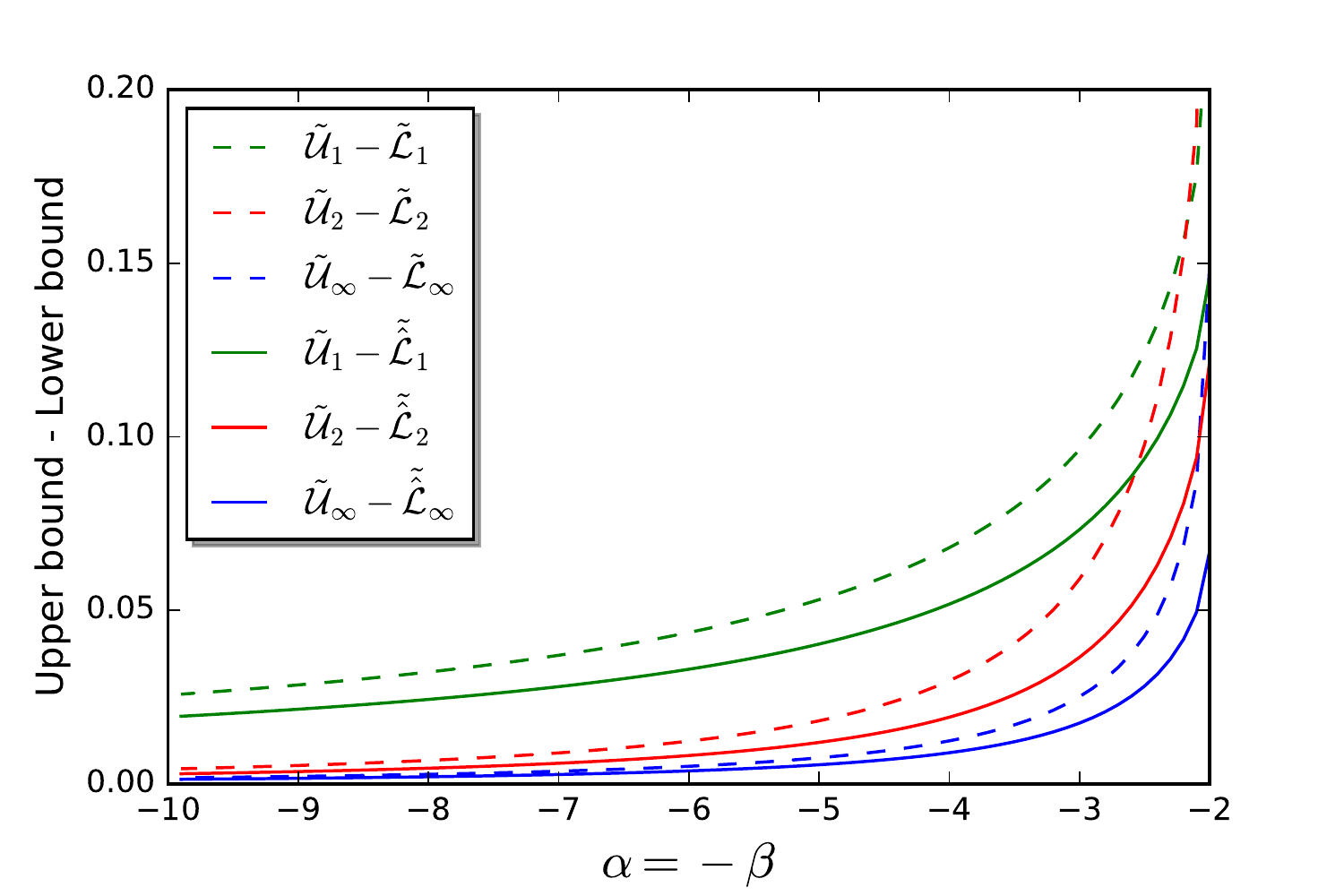}
 \caption{Envelope of bounds from Theorem~\ref{thm:global_cone_bounds_rev}, shown dashed, and from Theorem~\ref{thm:improved_cone_bounds_rev}, shown as solid lines.}\label{fig:reverse_env}
\end{subfigure}
\caption{Bounds for the negative entry case, in which $\alpha <-2, \beta>2$. In this case the $L_{\infty}$-norm always produces the most accurate bounds.}\label{fig:rev_figs}
\end{figure}

\begin{figure}
\captionsetup[subfigure]{justification=raggedright}
\centering
\begin{subfigure}[t]{0.49\linewidth}
\centering
 \includegraphics[width=\linewidth]{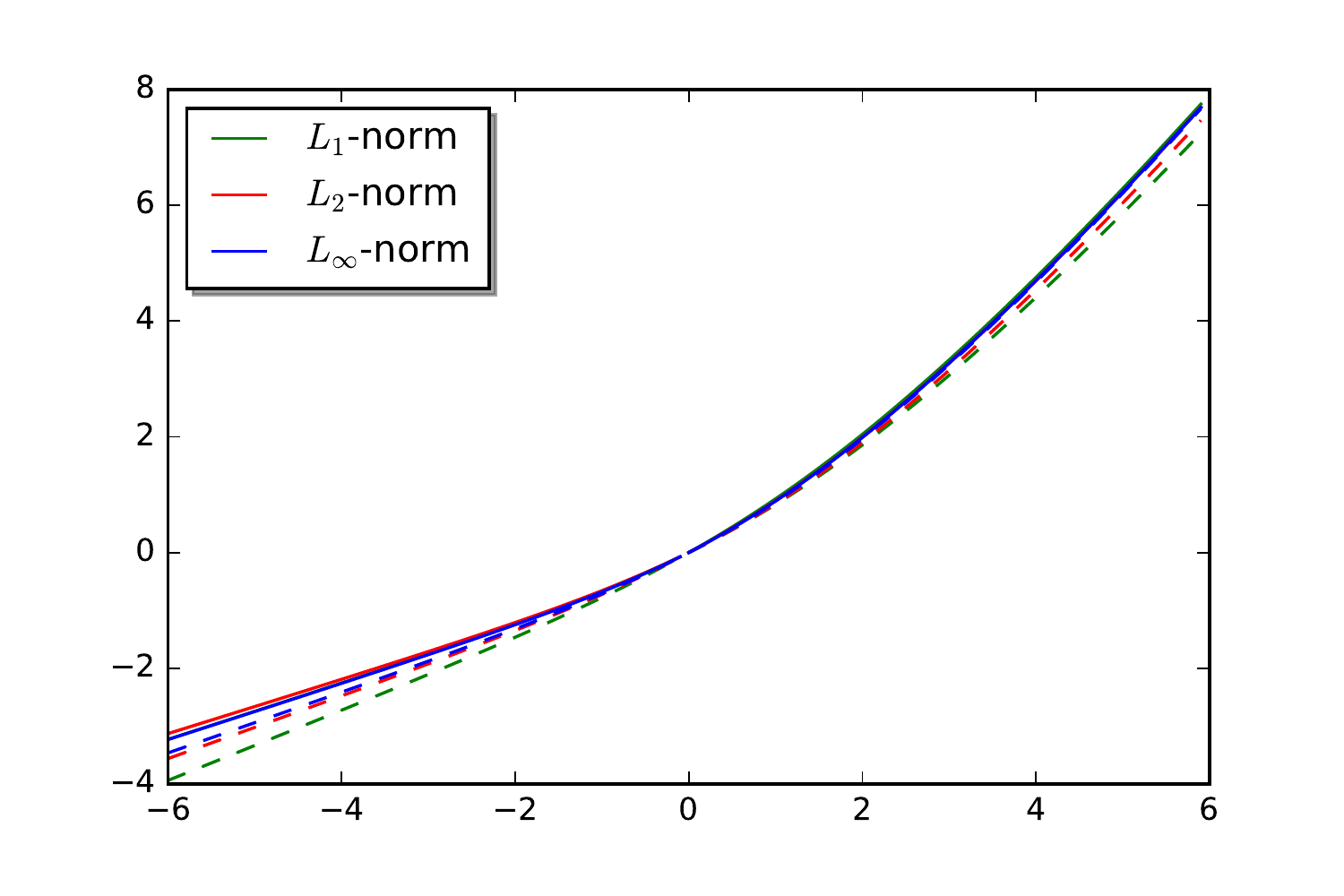}
 \caption{The curve of generalised Lyapunov exponents for $\alpha<0$, $\beta>0$.}\label{fig:gle_bounds_rev}
\end{subfigure}
\begin{subfigure}[t]{0.49\linewidth}
\centering
 \includegraphics[width=\textwidth]{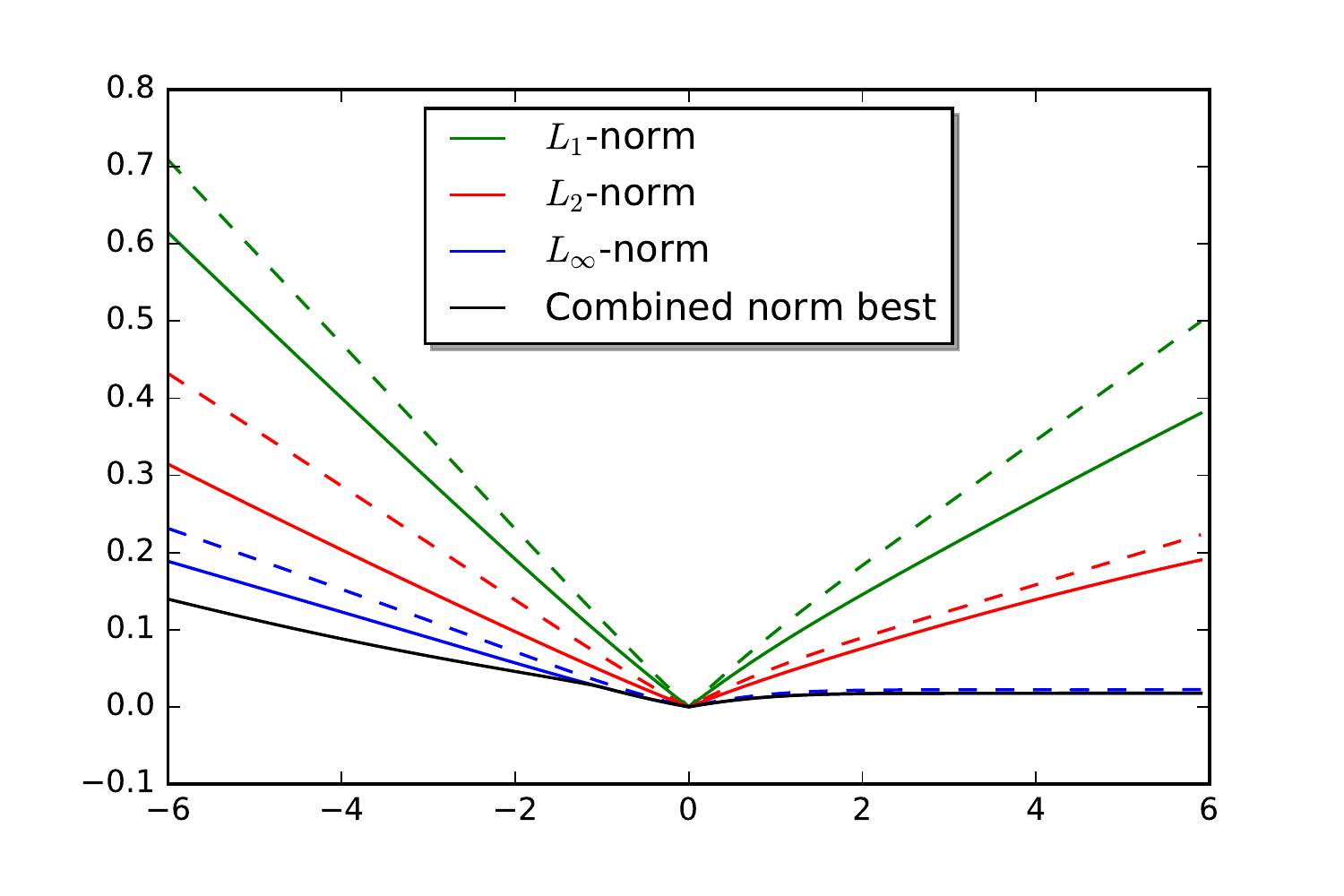}
 \caption{Difference between upper and lower bounds. Dashed lines represent bounds from Theorem~\ref{thm:global_cone_bounds_rev}, while solid lines are those from~\ref{thm:improved_cone_bounds_rev}. The black line represents the minimal envelope over all norms, and thus our best bounds.}\label{fig:gle_env_rev}
\end{subfigure}
\caption{Bounds for the generalised Lyapunov exponent for the matrix product \ref{eq:ABalphabeta} with $\alpha = -3$, $\beta =3$. As before, estimates originating from $L_1$-, $L_2$- and $L_{\infty}$-norms are given in green, red and blue respectively. }\label{fig:gle_figs_rev}
\end{figure}

\section{Bounds on the growth of matrix norms}

\subsection{Invariant cones}

We obtain bounds for Lyapunov exponents by computing explicit bounds for the norm of tangent vectors under the action of $\K_{ab}$. The expression (\ref{eq:lambdadef}) is independent of the matrix norm used, and we give bounds derived from three standard norms.

\begin{figure}
\captionsetup[subfigure]{justification=raggedright}
\centering
\begin{subfigure}[t]{0.49\linewidth}
\centering
 \includegraphics[width=0.7\linewidth]{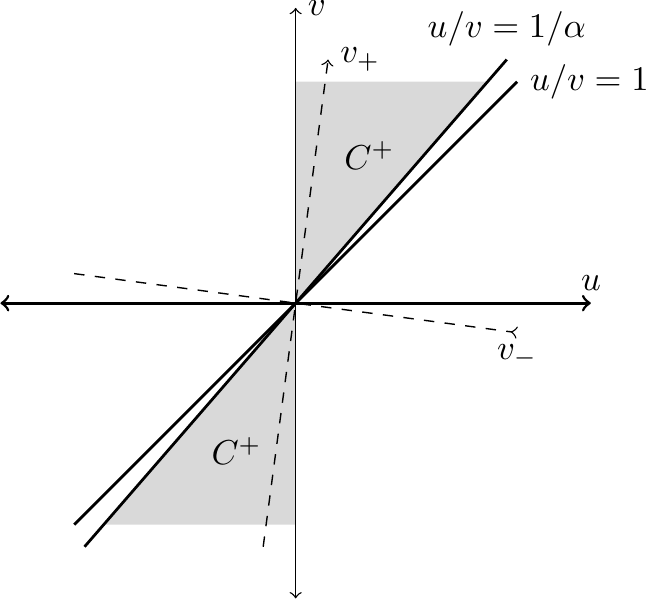}
 \caption{The $\alpha>0$ case. Here we show the cone $C^+$ for $\alpha>1$, where it lies inside the line $u/v = 1$. The expanding eigenvector $v_+$ also lies inside the cone $C^+$, and so the orthogonal contracting eigenvector $v_-$ lies outside $C^+$, and hence Lemma~\ref{lem:L2_bounds}.}\label{fig:cones1}
\end{subfigure}
\begin{subfigure}[t]{0.49\linewidth}
\centering
 \includegraphics[width=0.7\textwidth]{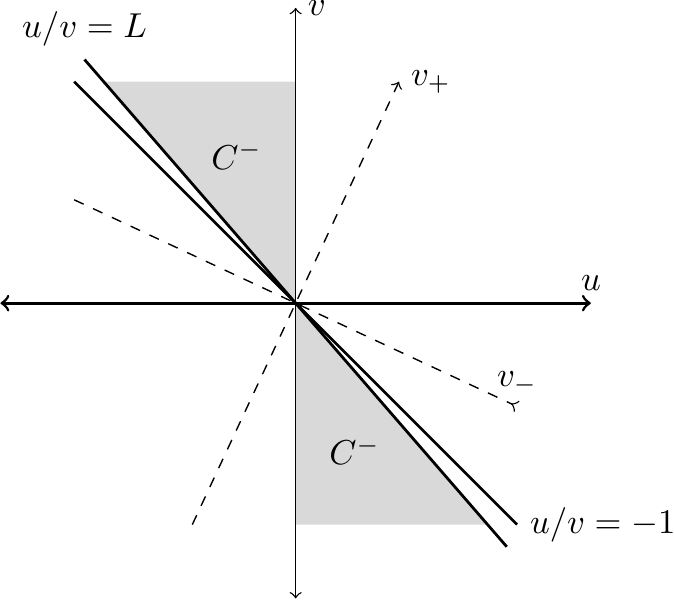}
 \caption{The $\alpha<0$ case. For $\alpha \beta <-4$, when the matrix $\K_{ab}$ is hyperbolic, the cone $C^-$ lies inside the line $u/v=L \in (-1,0)$. Both eigenvectors $v_+$ and $v_-$ both lie outside the invariant cone for all $\alpha<-2$,  $\beta<-2$, which produces Lemma~\ref{lem:L2bounds_rev}.}\label{fig:cones2}
\end{subfigure}
\caption{The invariant cones $C^+$ and $C^-$ in both the $\alpha>0$ and $\alpha<0$ cases, with expanding and contracting eigenvectors, $v_+$ and $v_-$ respectively, of the matrix $\K_{ab}^T \K_{ab}$.}\label{fig:cones}
\end{figure}

\begin{lemma}\label{prop:inv_cone}
  The cone~$C^+ = \{ (u,v) : 0 \le u/v \le 1/\alpha \}$ (shown in figure \ref{fig:cones1}) is invariant under~$\K_{ab}$ for all~$a,b \ge 1$, and is the smallest such cone.
\end{lemma}

\begin{proof}

The vector
$$
\begin{pmatrix} u' \\ v' \end{pmatrix}  = \pmat{1}{b\beta}{a\alpha}{1+a\alpha b\beta} \begin{pmatrix} u \\ v \end{pmatrix}
$$
is such that
$$
\frac{u'}{v'} = \frac{u+b\beta v}{a\alpha u+(1+a\alpha b\beta )v} \ge 0,
$$
clearly, and since $au+abv \ge u+bv$ for $a\ge 1$, we also have $u'/v' \le 1/\alpha$. Setting~$(a,b)=(1,\infty)$ gives $u'/v'=1/\alpha$, while setting~$(a,b)=(\infty,1)$ gives $u'/v'=0$, showing that the cone cannot be made smaller.

\end{proof}

Throughout this section, we will consider a  vector $X = (u,v)^T \in C^+$, and assume without loss of generality that $u,v \ge 1$ (the calculations for $u, v, <0$ are entirely analogous). We will consider the norm of the vector $\K_{ab}X$, given by
$$
\K_{ab} X =  \begin{pmatrix} u+b\beta v \\ a\alpha u+(1+a\alpha b\beta)v\end{pmatrix}.
$$
First we consider the $L_1$-norm, given  by $\lVert X \rVert_1 =  |u|+ |v|$.

\begin{lemma}\label{prop:L1_bounds}
The norm~$\lVert \K_{ab}X\rVert_1$ for a vector~$X \in C^+$ satisfies
\begin{enumerate}

\item[(i)] the lower bound
\begin{equation}
  \frac{\lVert \K_{ab} X\rVert_1}{\lVert X\rVert_1}
  \ge 1+\tfrac{\alpha}{1+\alpha} \left( a+ b\beta + a\alpha b\beta \right);
\label{eq:MX_L1_lower}
\end{equation}

\item[(ii)] the upper bound
\begin{equation}
  \frac{\lVert \K_{ab} X\rVert_1}{\lVert X\rVert_1}   \le
  1+b\beta+a\alpha b\beta.
\label{eq:MX_L1_upper}
\end{equation}

\end{enumerate}
\end{lemma}

\begin{proof}
For any $X \in C^+$ we have $\lVert X \rVert_1 = u+v$. With $a,b \ge 1$ we have
$$
\frac{\lVert \K_{ab} X \rVert_1}{\lVert  X \rVert_1} = 1+\frac{b\beta v + a \alpha u + a\alpha b \beta v}{u+v}\,.
$$
With $\alpha, \beta \ge 1$ this has no local minima or maxima in the cone $C^+$. Thus the lower and upper bounds are attained at the boundaries of $C^+$, given by $(u,v) = (\frac{1}{1+\alpha},\frac{\alpha}{1+\alpha})$ and $(u,v) = (0,1)$ respectively.
\end{proof}

For the $L_2$-norm $\lVert X \rVert_2 = \sqrt{u^2+v^2}$, the calculations are more involved, but the following holds:

\begin{lemma}
\label{lem:L2_bounds}
The norm~$\lVert \K_{ab}X\rVert_2$ for a vector~$X \in C^+$ satisfies
\begin{enumerate}

\item[(i)] the lower bound
\begin{equation}
  \frac{\lVert \K_{ab} X\rVert_2^2}{\lVert X\rVert_2^2}
  \ge \min\left\{
    (1+a\alpha b\beta)^2 + b^2\beta^2\,,\,
    \tfrac1{1+\alpha^2}\left(\alpha^2(1+a+a\alpha b\beta)^2 + (1+\alpha b\beta)^2\right)\right\};
\label{eq:MXlower}
\end{equation}

\item[(ii)] the upper bound
\begin{equation}
  \frac{\lVert \K_{ab} X\rVert_2^2}{\lVert X\rVert_2^2} \le
  \tfrac12 \left(2 + \mathcal{C}_{a\alpha b\beta}
  + \sqrt{\mathcal{C}_{a\alpha b\beta} (\mathcal{C}_{a\alpha b\beta} + 4)}\right),
  \qquad
  \mbox{where } \mathcal{C}_{a\alpha b \beta} \ldef (a\alpha+b\beta)^2 + (a\alpha b \beta)^2.
\label{eq:MXupper}
\end{equation}

\end{enumerate}
\end{lemma}

\begin{proof}
The real $2\times 2$ matrix $\K_{ab}$ is non-singular, and so $\forall X \in \mathbb{R}^2$, $\frac{\lVert \K_{ab} X\rVert_2}{\lVert X\rVert_2}$ is maximised (minimised) by $\frac{\lVert \K_{ab} v_+ \rVert_2}{\lVert v_+\rVert_2}$ $\left(\frac{\lVert \K_{ab} v_- \rVert_2}{\lVert v_-\rVert_2}\right)$, where $v_+$ ($v_-$) is the eigenvector corresponding to $e_+$ ($e_-$), the larger (smaller) eigenvalue of $\K_{ab}^T\K_{ab}$, by the definition of the spectral matrix norm (and by singular value decomposition). Moreover the value of $\frac{\lVert \K_{ab} X\rVert_2}{\lVert X\rVert_2}$ varies monotonically between these extremes. Since $\K_{ab}^T\K_{ab}$ is symmetric, $v_-$ and $v_+$ are orthogonal.

The eigenvector $v_+ = (r,s)$ satisfies
\begin{equation}\label{eq:MTMev}
\frac{r}{s} = \frac{2(a\alpha(a\alpha b\beta+1)+b\beta)}{\mathcal{C}_{a\alpha b\beta} - 2a^2\alpha^2 + \sqrt{\mathcal{C}_{a\alpha b\beta} (\mathcal{C}_{a\alpha b\beta} + 4)}}\,.
\end{equation}
Clearly $r>0$, while  $s=\mathcal{C}_{a\alpha b\beta} - 2a^2\alpha^2 + \sqrt{\mathcal{C}_{a\alpha b\beta} (\mathcal{C}_{a\alpha b\beta} + 4)} > 2C_{a\alpha b\beta} - 2a^2\alpha^2 = 2b^2\beta^2 + 4a\alpha b\beta + 2(a\alpha b\beta)^2 >0$, and so $v_+$ lies in the positive quadrant of tangent space. Moreover, we have $s > 2\mathcal{C}_{a\alpha b\beta} - 2a^2\alpha^2 = 4a\alpha b\beta + 2b^2\beta^2 + 2a\alpha b\beta)^2 > 2a\alpha + 2b\beta + 2a^2\alpha^2 b\beta = r$ (since $a, b, \alpha, \beta \ge 1$), and so $v_+ \in C^+$, giving the upper bound.  The orthogonality of the eigenvectors then implies that $v_- \notin C^+$, and the lower bound is given by the minimum of the value of the spectral norm on the boundaries of $C^+$.

\end{proof}

Next, consider the $L_\infty$-norm, given by $\lVert X \rVert_{\infty} = \max (|u|, |v|).$

\begin{lemma}\label{prop:inf_bounds}
The norm~$\lVert \K_{ab}X\rVert_{\infty}$ for a vector~$X \in C^+$ satisfies, for $\alpha \ge1$,
\begin{enumerate}

\item[(i)] the lower bound
\begin{equation}
  \frac{\lVert \K_{ab} X\rVert_{\infty}}{\lVert X\rVert_{\infty}}
  \ge 1+a\alpha b\beta\,;
\label{eq:MX_inf_lower}
\end{equation}

\item[(ii)] the upper bound
\begin{equation}
  \frac{\lVert \K_{ab} X\rVert_{\infty}}{\lVert X\rVert_{\infty}}   \le
 1+a+a\alpha b\beta\,.
\label{eq:MX_inf_upper}
\end{equation}

\end{enumerate}
\end{lemma}

\begin{proof}
For $\alpha \ge 1$ we have $\lVert X \rVert_{\infty} = v$. Then
$$
\frac{\lVert \K_{ab} X \rVert_{\infty}}{\lVert X \rVert} = a\alpha \tfrac{u}{v}+(1+a\alpha b\beta ).
$$
This takes minimum and maximum values at minimum and maximum values of $u/v$, respectively. For the cone $C^+$ these are given by 0 and $1/\alpha$, and the bounds follow immediately.
\end{proof}

We can now use these bounds and the invariant cone to prove Theorem \ref{thm:global_cone_bounds}.

\begin{proof}[Proof of Theorem \ref{thm:global_cone_bounds}]

Taking i.i.d.\ copies of~$\K_{ab}$ and defining~$X_{N_k} = \K_{a_kb_k}X_{N_{k-1}}$,
$k=1,\ldots,J$, we have for an initial vector $X_0 \in C^+$,
\begin{equation}
  \lVert X_{N_J}\rVert =
  \frac{\lVert \K_{a_Jb_J} X_{N_{J-1}} \rVert}{\lVert X_{N_{J-1}} \rVert}\,
  \frac{\lVert \K_{a_{J-1}b_{J-1}} X_{N_{J-2}} \rVert}{\lVert X_{N_{J-2}} \rVert}
  \cdots
  \frac{\lVert \K_{a_1b_1} X_{0} \rVert}{\lVert X_{0} \rVert}.
\end{equation}
By Lemma \ref{prop:inv_cone}, each term in the product is a vector $\in C^+$, and so is bounded according to Lemmas \ref{prop:L1_bounds}, \ref{lem:L2_bounds} and \ref{prop:inf_bounds}. Hence
$$
(\phi_k(a,b,\alpha,\beta))^J \le \lVert X_{N_J}  \rVert \le  (\psi_k (a,b,\alpha,\beta))^J
$$
for $k = 1, 2, \infty$. Now
$$
\lambda = \lim_{N \to \infty} \frac{1}{N} \E \log \lVert X_N \rVert = \lim_{J   \to \infty} \frac{1}{4J} \E \log \lVert X_{N_J} \rVert
$$
since $\E n = 4$, and so using the probability distribution $P(a,b)$ we have
$$
\lim_{J \to \infty} \frac{1}{J}\sum_{a, b=1}^{\infty}2^{-a-b} \log (\phi_k (a,b,\alpha,\beta))^J \le 4\lambda \le \lim_{J \to \infty} \frac{1}{J} \sum_{a,b=1}^{\infty} 2^{-a-b} \log (\psi_k (a,b,\alpha,\beta))^J
$$
and hence
$$
\sum_{a,b=1}^{\infty}2^{-a-b} \log \phi_k (a,b,\alpha,\beta) \le 4\lambda \le \sum_{a,b=1}^{\infty} 2^{-a-b} \log \psi_k (a,b,\alpha,\beta) \\
$$
as required.
\end{proof}

To obtain Corollary \ref{cor:explicit_bounds} we select the algebraically simplest bounds (the $L_{\infty}$ bounds), and evaluate the infinite sums where possible.
\begin{proof}[Proof of Corollary \ref{cor:explicit_bounds}]
The lower $L_{\infty}$ bound immediately gives:
\begin{eqnarray*}
4\lambda &\ge& \sum_{a,b=1}^{\infty} 2^{-a-b} \log (1+a\alpha b\beta) \\
&\ge& \sum_{a,b=1}^{\infty} 2^{-a-b} \log (a\alpha b\beta) \\
&=& \sum_{a,b=1}^{\infty} 2^{-a-b} \log a b + \log (\alpha \beta) \sum_{a,b=1}^{\infty} 2^{-a-b}  \\
&=& \kappa + \log \alpha \beta.
\end{eqnarray*}
A little more work is required for the upper bound. We have
\begin{eqnarray*}
4\lambda &\le & \sum_{a,b=1}^{\infty} 2^{-a-b} \log (1+a+ a\alpha b\beta) \\
&=& \sum_{a=1}^{\infty} 2^{-a-1}\log (1+a+a\alpha\beta) + \sum_{a=1}^{\infty} \sum_{b=2}^{\infty} 2^{-a-b} \log (1+a+a\alpha b\beta)\\
&\le& \tfrac{1}{2}\sum_{a=1}^{\infty} 2^{-a}\log \left( a( \sqrt{\alpha \beta} + 1/{\sqrt{\alpha \beta}}  )^2  \right) + \sum_{a=1}^{\infty} \sum_{b=2}^{\infty} 2^{-a-b} \log (ab(1+\alpha \beta))
\end{eqnarray*}
since $a\left( \sqrt{\alpha \beta} + {1}/{\sqrt{\alpha \beta}}  \right)^2 = a(\alpha \beta + 2 + 1/\alpha \beta) > a\alpha \beta + a+1$ for $a\ge 1$, and since $ab(1+\alpha \beta)>1+a+a\alpha b\beta$ for $b \ge 2$. Then the logarithms can be separated, reinstating and subtracting the $b=1$ term to the second term, to give
$$
4\lambda \le \tfrac{1}{2}\sum_{a=1}^{\infty} 2^{-a} \log a + \log (\sqrt{\alpha \beta} + 1/\sqrt{\alpha \beta}) + \kappa
+ \log (1+\alpha \beta) - \sum_{a=1}^{\infty}2^{-a-1} \log a(1+\alpha \beta)
$$
and hence
$$
4\lambda \le  \kappa + \log (\sqrt{\alpha \beta} + 1/\sqrt{\alpha \beta}) + \tfrac{1}{2} \log (1+\alpha \beta).
$$

\end{proof}

\subsection{Cone improvement}

In this section we improve on the lower bound by considering the relationship between two identical geometric distributions.

\begin{lemma}\label{lem:geom}
When $a$ and $b$ are both i.i.d.\ geometric distributions with parameter $1/2$, we have
$$
P(a=b) = P(a>b) = P(b>a) = 1/3.
$$
\end{lemma}

\begin{proof}
We have
$$
P(a=b) = \sum_{i=1}^{\infty} P(a=i \cap b=i) = \sum_{i=1}^{\infty} 2^{-2i} = \frac{1/4}{1-1/4} = \frac{1}{3}.
$$
Then the remaining two equalities follow by symmetry.
\end{proof}

\begin{lemma}\label{lem:cone_improve}
The cone $C = \{ 0 \le \frac{u}{v} \le \frac{1}{\alpha}\}$ is mapped into the following cones, in the following cases:
\begin{enumerate}[1.]
\item when $a<b$, $\K_{ab} (C) = C$;
\item when $a=b$, $\K_{ab} (C) = \{ 0 \le \frac{u}{v} \le \frac{1+\alpha \beta}{2\alpha+\alpha^2 \beta} \}$;
\item when $a>b$, $\K_{ab} (C) = \{ 0 \le \frac{u}{v} \le \frac{1+\alpha \beta}{3\alpha + 2\alpha^2 \beta} \}$.
Consequently, we have
$$
\phi_k^{(m)}(a,b,\alpha,\beta) \le \frac{\lVert \K_{ab} X \rVert }{\lVert X\rVert } \le \psi_k^{(m)}(a,b,\alpha,\beta),
$$
for $k=1, 2, \infty$, and for $m=1,2,3$ corresponding to the cases above, with $\phi_k^{(m)}$ and $\psi_k^{(m)}$ as given in theorem \ref{thm:improved_cone_bounds}.

\end{enumerate}
\end{lemma}

\begin{proof}
In each case, the cone boundary $(0,1)^T$ is mapped onto $(b\beta,1+a\alpha b\beta)^T$, which lies arbitrarily close to $(0,1)^T$ for large $a$, regardless of the relationship between $a$ and $b$, and for all $\alpha$, $\beta>0$. The other cone boundary $(1,\alpha)^T$ is mapped onto $(1+\alpha b\beta,1+ a\alpha+a\alpha b\beta)^T$, and  then we observe that:
\begin{enumerate}[1.]
\item if $a=b$, the ratio $\frac{1+a\alpha \beta}{a\alpha+\alpha+a^2 \alpha^2 \beta}$ is maximised when $a=1$;
\item if $a>b$, the ratio $\frac{1+b\alpha \beta}{a\alpha +\alpha +a\alpha^2 b\beta}$ is maximised when $a=2$ and $b=1$;
\item if $b>a$, the ratio $\frac{1+\alpha b\beta}{a\alpha +\alpha +a\alpha^2 b \beta}$ approaches $\frac{1}{\alpha}$ for $a=1$ and $b \to \infty$.
\end{enumerate}
The bounds then follow using the same derivations as in Lemmas \ref{prop:L1_bounds}, \ref{lem:L2_bounds} and \ref{prop:inf_bounds}, substituting these new cone boundaries where appropriate.
\end{proof}

\begin{proof}[Proof of Theorem \ref{thm:improved_cone_bounds}]

This follows the same argument as the proof of theorem \ref{thm:global_cone_bounds}, except that whenever it happens that $a=b$, or $a>b$, on the {\em following} iterate the vector $\lVert X_i \rVert$ is bounded according to Lemma \ref{lem:cone_improve}. Since by Lemma \ref{lem:geom} these conditions occur on average $1/3$ of the time, the result follows.
\end{proof}

\subsection{Negative shears}

As in section \ref{sec:neg} we reverse one of the shears, taking (without loss of generality) $\alpha <-2, \beta >2$, with $a, b >0$.  Eigenvalues of $\K_{ab}$ are then given by
$$
e_{\pm} = \frac{2+a\alpha b\beta \pm \sqrt{a\alpha b \beta (a\alpha b\beta +4)}}{2}\,.
$$
The expanding eigenvalue $e_-$ has eigenvector $(u,v)^T$ with
$$
\frac{u}{v} = -\frac{b\beta}{2} + \sqrt{\left( \frac{b\beta}{2}  \right)^2 + \frac{b\beta}{a\alpha}} < 0.
$$
In the case $\alpha <-2$ the minimal cone is bounded by this eigenvector when $a=b=1$, so setting
$$
\Gamma = -\frac{\beta}{2} + \sqrt{\left( \frac{\beta}{2}  \right)^2 + \frac{\beta}{\alpha}} \in (-1,0)
$$
we have:
\begin{lemma}\label{lem:neg_cone}
The cone $C^{-} = \{ (u,v) : \Gamma \le u/v \le 0 \}$ is invariant under $\K_{ab}$ for all $a, b \ge 1$, and for all $\alpha <-2$, $\beta>2$, and is the smallest such cone.
\end{lemma}

\begin{proof}
Without loss of generality we will take an initial vector $(u,v)$ with $u<0, v>0$ (an initial vector in the opposite sector proceeds exactly analogously) in $C^-$, so that $-u<v$ and $u>-v$. Then we consider
$$
\begin{pmatrix} u' \\ v' \end{pmatrix}  = \pmat{1}{b\beta}{a\alpha}{1+a\alpha b\beta} \begin{pmatrix} u \\ v \end{pmatrix} = \begin{pmatrix} u+b\beta v \\ a\alpha u+(1+a\alpha b\beta )v \end{pmatrix}.
$$
Now $u' = u+b\beta v >v(b\beta-1)>0$, and $v' = a\alpha u+(1+a\alpha b\beta )v < a\alpha u +u(-1-a\alpha b\beta) = u(-a\alpha(b\beta-1)-1)<0$, and so $u'/v' <0$.

Since $e_- = 1+\beta/\Gamma$, the characteristic equation for $\K_{11}$ is $\alpha \Gamma^2 + \alpha \beta \Gamma - \beta  =0$. Then since $\alpha \beta \Gamma > |\alpha \Gamma^2|$ (since $\beta > 2 > |\Gamma|$) we have $a\alpha \Gamma^2 + a\alpha \beta \Gamma - \beta  \ge 0$ for $a\ge 1$. We also have $a\alpha \beta \Gamma > \beta$, and so for $b \ge 1$,
$$
a \alpha \Gamma^2 + a \alpha b \beta \Gamma - b \beta  \ge 0
$$
and hence
$$
a \alpha \Gamma^2 + a \alpha b \beta \Gamma   +\frac{u}{v} \ge b\beta + \frac{u}{v}.
$$
Now we use the fact that $|a\alpha u/v| > |u/v| $ to replace two of these terms while respecting the inequality:
$$
a \alpha \Gamma\frac{u}{v} + a \alpha b \beta \Gamma   + \Gamma \ge b\beta + \frac{u}{v}.
$$
Rearranging then gives $u'/v' \ge \Gamma$. This is the smallest such invariant cone, since setting $(a,b) = (1,1)$ gives $u'/v' = \Gamma$ when $u/v = \Gamma$, and setting $(a,b) = (\infty,1)$ gives $u'/v' = 0$.
\end{proof}

As before the $L_{\infty}-$norm gives bounds easily:
\begin{lemma}\label{prop:inf_bounds_rev}
The norm~$\lVert \K_{ab}X\rVert_{\infty}$ for a vector~$X \in C^-$, when $\alpha<-2, \beta>2$,  satisfies
\begin{enumerate}

\item[(i)] the lower bound
\begin{equation}
  \frac{\lVert \K_{ab} X\rVert_{\infty}}{\lVert X\rVert_{\infty}}
  \ge -a\alpha b\beta - a\alpha \Gamma -1;
\label{eq:MX_inf_lower_rev}
\end{equation}

\item[(ii)] the upper bound
\begin{equation}
  \frac{\lVert \K_{ab} X\rVert_{\infty}}{\lVert X\rVert_{\infty}}   \le
  -a\alpha b\beta -1.
\label{eq:MX_inf_upper_rev}
\end{equation}

\end{enumerate}
\end{lemma}

\begin{proof}
Since $\Gamma>-1$, for any $X \in C^-$ we have $\lVert X \rVert_{\infty} = |v|$ and since $C^-$ is invariant under $\K_{ab}$, we have $\lVert \K_{ab} X \rVert_{\infty} / \lVert  X \rVert_{\infty} = |a\alpha \frac{u}{v}+1+a\alpha b\beta |$, which takes minimum and maximum values at the boundaries $(u,v) = (0,1)$ and $(u,v) = (\Gamma,1)$ of the cone $C^-$, and the bounds follow immediately.
\end{proof}

For this invariant cone, the $L_2$-norm $\lVert \cdot \rVert_2$ cannot attain the spectral maximum, and the following holds:
\begin{lemma}
\label{lem:L2bounds_rev}
The norm~$\lVert \K_{ab}X\rVert_2$ for a vector~$X \in C^-$ satisfies
\begin{enumerate}

\item[(i)] the lower bound
\begin{equation}
  \frac{\lVert \K_{ab} X\rVert_2^2}{\lVert X\rVert_2^2}
  \ge \tfrac1{1+\Gamma^2}\left( (\Gamma+b\beta)^2 + (1+a\alpha \Gamma+a\alpha b\beta)^2\right);
\end{equation}

\item[(ii)]  the upper bound
\begin{equation}
  \frac{\lVert \K_{ab} X\rVert_2^2}{\lVert X\rVert_2^2}
  \le
    (1+ab\alpha \beta)^2 + b^2\beta^2\,.
\end{equation}
.

\end{enumerate}
\end{lemma}

\begin{proof}
As in  Lemma \ref{lem:L2_bounds}, we consider eigenvectors of $\K_{ab}^T\K_{ab}$. For $\alpha<-2$, $\beta>2$, the expanding eigenvector $v_+ = (r,s)$ still lies in the northeast-southwest quadrant, outside $C^-$. But since $v_- = (-s,r)$,  we have
\begin{eqnarray*}
s &=& C_{a\alpha b\beta} -2a^2\alpha^2+\sqrt{C_{a\alpha b\beta}(C_{a\alpha b\beta}+4)} \\
&>& 2C_{a\alpha b\beta} - 2a^2\alpha^2 \\
&>& 4a\alpha b\beta +2 b^2\beta^2 +2^2a\alpha^2 b^2\beta^2 \\
&>& 2a\alpha + 2b\beta +2a^2\alpha^2b\beta \\
&=& r,
\end{eqnarray*}
and so $-s/r <-1$, and hence $v_-$ also lies outside $C^-$. Since the norm in question increases monotonically between the two extremes, neither of which lie in the cone, the lower and upper bounds are achieved at the minimum and maximum values (respectively) at the boundaries of $C^-$. At the boundary given by $(u,v)=(0,1)$, we have  $\frac{\lVert \K_{ab} X\rVert_2^2}{\lVert X\rVert_2^2} = b^2\beta^2 +(1+a\alpha b\beta)^2$, while at the other boundary, given by $(u,v) = (\Gamma/\sqrt{1+\Gamma^2},1/\sqrt{1+\Gamma^2})$, we have
\begin{eqnarray*}
\frac{\lVert \K_{ab} X\rVert_2^2}{\lVert X\rVert_2^2} &=& \tfrac{1}{1+\Gamma^2} \left( (\Gamma+b\beta)^2 + (1+a\alpha \Gamma +a\alpha b\beta)^2   \right)  \\
& <& (\Gamma+b\beta)^2 + (1+a\alpha \Gamma +a\alpha b\beta)^2 \\
&<& b^2\beta^2 +(1+a\alpha b\beta)^2,
\end{eqnarray*}
since $-1<\Gamma<0$.
\end{proof}

\begin{lemma}\label{lem:L1_bounds_rev}
The norm~$\lVert \K_{ab}X\rVert_1$ for a vector~$X \in C^-$ satisfies
\begin{enumerate}

\item[(i)] the lower bound
\begin{equation}
  \frac{\lVert \K_{ab} X\rVert_1}{\lVert X\rVert_1}
  \ge \tfrac{1}{1-\Gamma} (b\beta - a\alpha b \beta - 1 -\Gamma(a\alpha + 1));
\label{eq:MX_L1_lower_rev}
\end{equation}

\item[(ii)] the upper bound
\begin{equation}
  \frac{\lVert \K_{ab} X\rVert_1}{\lVert X\rVert_1}   \le
  b\beta - a\alpha b\beta  -1\,.
\label{eq:MX_L1_upper_rev}
\end{equation}

\end{enumerate}

\end{lemma}

\begin{proof}
With the $L_1$-norm we have $\lVert \K_{ab} X \rVert_1 = |u+b\beta v| + |-a\alpha u +(1+a\alpha b\beta)v|$, which takes the given values at the boundaries $(u,v) = (0,1)$ and $(u,v) =  (\Gamma/(1-\Gamma),1/(1-\Gamma)$ of $C^-$.
\end{proof}

\begin{proof}[Proof of Theorem~\ref{thm:global_cone_bounds_rev}]
This follows exactly the argument of Theorem \ref{thm:global_cone_bounds}, using Lemma \ref{lem:neg_cone} to guarantee an invariant cone, and using Lemmas \ref{prop:inf_bounds_rev}, \ref{lem:L2bounds_rev} and \ref{lem:L1_bounds_rev} to bound each term in the matrix product.
\end{proof}

\subsection{Cone improvement}

In the $\alpha<0$ case we can make a significant improvement on the bounds given by Theorem \ref{thm:global_cone_bounds_rev} by recognising that the boundary $u/v = \Gamma$ of the cone $C^-$ can only be achieved when $a=b=1$, which occurs on average $P(a=b=1) = 1/4$ of the time. Whenever $a$ or $b$ (or both) is greater than 1, we can assume a smaller cone for the following iterate. More precisely, since $P(a=1, b\ge 2) = P(a\ge 2, b=1)= P(a\ge 2, b\ge2) = 1/4$, we have

\begin{lemma}\label{lem:cone_improve_rev}
The cone $C^- = \{ \Gamma \le \frac{u}{v} \le 0\}$ is mapped into the following cones with equal probability:
\begin{enumerate}
\item When $a=b =1$, $\K_{ab} (C^-) = \{ \Gamma \le \frac{u}{v} \le \frac{\beta}{1+\alpha \beta}   \}$;
\item when $a\ge 2, b = 1$, $\K_{ab} (C^-) = \{ \Gamma_{2,1} \le \frac{u}{v} \le 0 \}$;
\item when $a = 1, b\ge 2$, $\K_{ab} (C^-) = \{ \Gamma_{1,2} \le \frac{u}{v} \le \frac{1}{\alpha} \}$;
\item when $a \ge 2, b\ge 2$, $\K_{ab} (C^-) = \{ \Gamma_{2,2} \le \frac{u}{v} \le 0 \}$.
\end{enumerate}
These cones then produce the functions $\hat{\tilde{\phi}}_k^{(m_a,m_b)}(a,b,\alpha,\beta)$ and $\hat{\tilde{\psi}}_{\infty}^{(m)} (a,b,\alpha,\beta)$, for $m_a, m_b = 1,2$ and $m = 1,2,3$ as detailed in Theorem \ref{thm:improved_cone_bounds_rev}.
\end{lemma}

\begin{proof}
Any vector $(u,v)$ is mapped by $\K_{ab}$ into $(u',v')$ such that
$$
\frac{u'}{v'} = \frac{\frac{u}{v}+b\beta}{a\alpha \frac{u}{v}+1+a\alpha b\beta}\,.
$$
Inserting the boundaries of $C^-$, given by $\frac{u}{v} = 0$ and $\frac{u}{v} = \Gamma$  into this expression produces the required inequalities. The bounding functions are then obtained using analogous arguments to Lemmas \ref{prop:inf_bounds_rev}, \ref{lem:L2bounds_rev} and \ref{lem:L1_bounds_rev}, with the new cone boundaries.
\end{proof}

\begin{proof}[Proof of Theorem \ref{thm:improved_cone_bounds_rev}]

Again this follows the same argument as the proof of theorem \ref{thm:global_cone_bounds}, using improved bounds given by Lemma \ref{lem:cone_improve_rev}, each of which applies $1/4$ of the time, on average.
\end{proof}

\subsection*{Generalised Lyapunov exponents}

The expressions for $\ell(q)$ can be obtained in largely the same way, bounding the expansion of vectors at each application of matrix $A$ or $B$.

\begin{proof}[Proof of Theorems~\ref{thm:gle_bounds} and \ref{thm:gle_bounds_rev}]
Using properties of expectation, and the independence of $\lVert X_i \rVert$, we have
\begin{eqnarray*}
\E \lVert X_{N_J} \rVert^q &=& \E \left(  \frac{\lVert \K_{a_Jb_J} X_{N_{J-1}} \rVert}{\lVert X_{N_{J-1}} \rVert}\,
  \frac{\lVert \K_{a_{J-1}b_{J-1}} X_{N_{J-2}} \rVert}{\lVert X_{N_{J-2}} \rVert}
  \cdots
  \frac{\lVert \K_{a_1b_1} X_{0} \rVert}{\lVert X_{0} \rVert} \right)^q  \\
&=& \E \left( \frac{\lVert \K_{a_Jb_J} X_{N_{J-1}} \rVert}{\lVert X_{N_{J-1}} \rVert} \right)^q\,
  \E \left( \frac{\lVert \K_{a_{J-1}b_{J-1}} X_{N_{J-2}} \rVert}{\lVert X_{N_{J-2}} \rVert} \right)^q
  \cdots
 \E\left( \frac{\lVert \K_{a_1b_1} X_{0} \rVert}{\lVert X_{0} \rVert} \right)^q
\end{eqnarray*}
and so since the $a_i, b_i$ are i.i.d., we have
$$ \sum_{a,b=1}^{\infty}  2^{-a-b} \phi^q \le \E \lVert X_{N_J} \rVert^q \le  \sum_{a,b=1}^{\infty} 2^{-a-b} \psi^q\,.
$$
Then from the definition of $\ell(q)$ given in (\ref{eq:gle2}) the results follow immediately.
\end{proof}

\section{Conclusions and discussion}

In this paper we addressed the question of obtaining rigorous bounds for Lyapunov exponents, generalised Lyapunov exponents, and topological entropy for randomised mixing devices. The matrices under discussion are $2 \times 2$ shear matrices, but the same technique will work for any set of matrices that share an invariant cone. This notion is proved formally in \cite{protasov2013lower}, who give a rapid algorithm involving unconstrained minimisation problems. Here the optimisation is achieved analytically, giving explicit upper and lower bounds. We also obtain bounds in the novel case of shear matrices with negative entries. A pair of hyperbolic matrices sharing an invariant cone was shown to enjoy exponential decay of correlations in \cite{ayyer_exponential_2007}, where the rate of decay depends on the Lyapunov exponent, but here the Lyapunov exponent is simply bounded from below by global expansion and contraction rates in the invariant cone. The method in this paper could be adapted to tighten their lower bound, and provide an upper bound.

The assumption that the matrices $A$ and $B$ should be chosen with equal probability at each iterate can be relaxed. Altering these probabilities does not change the invariant cone, or the resulting bounds on vector norms; only the probability distribution $P(a,b) = 2^{-a-b}$  is changed. For example, replacing the geometric probability distribution with a Bernoulli distribution gives $P(a,b) = p^aq^b$, and  then~$\E a=q^{-1}$, $\E b=p^{-1}$, and $\E n = (pq)^{-1}$. Similarly, one may choose from $k$ matrices $A_i$ with probability $p_i$ at each iterate. The crucial element is that the expected length of a block should be computable.

Theorem \ref{thm:improved_cone_bounds} improves on \ref{thm:global_cone_bounds} by involving the relative values of $a$ and $b$ in one block to shrink the cone for the next, in the three cases $a=b$, $a<b$ and $a>b$. Similarly, the nine cases comprising the relative values of~$a$ and $b$ in two consecutive blocks can increase the  tightness of  bounds in the following block. This procedure could be extended to further improve bounds, but the number of cases increases exponentially --- in $k$ blocks there are $3^k$ combinations of relative values of $a$ and $b$. Our original explicit bounds are appealing in their simplicity and accuracy.

\ack

The authors thank Marko Budi\v{s}i\'{c} and Jacques Vanneste for helpful discussions.  This work began while the authors were visiting Trinity College, Cambridge. Visits between the two authors were supported by a grant from the University of Leeds Worldwide Universities Network Fund for International Research Collaboration.  J-LT was partially supported by NSF grant CMMI-1233935.

\section*{References}

\bibliographystyle{iopart-num}
\bibliography{hypshears}

\providecommand{\newblock}{}
\begin{thebibliography}{10}
\expandafter\ifx\csname url\endcsname\relax
  \def\url#1{{\tt #1}}\fi
\expandafter\ifx\csname urlprefix\endcsname\relax\def\urlprefix{URL }\fi
\providecommand{\eprint}[2][]{\url{#2}}

\bibitem{crisanti_products_1993}
Crisanti A, Paladin G and Vulpiani A 1993 {\em Products of random matrices in
  statistical physics\/} (Berlin ; New York: Springer)

\bibitem{heyde1985confidence}
Heyde C and Cohen J~E 1985 {\em Theoretical Population Biology\/} {\bf 27}
  120--153

\bibitem{bougerol_products_1985}
Bougerol P and Lacroix J 1985 {\em Products of random matrices with
  applications to {Schrödinger} operators\/} ({\em Progress in probability and
  statistics\/} vol~8) (Boston: Birkhäuser)

\bibitem{bellman1954limit}
Bellman R 1954 {\em Duke Mathematical Journal\/} {\bf 21} 491--500

\bibitem{furstenberg_products_1960}
Furstenberg H and Kesten H 1960 {\em The Annals of Mathematical Statistics\/}
  {\bf 31} 457--469

\bibitem{furstenberg_noncommuting_1963}
Furstenberg H 1963 {\em Transactions of the American Mathematical Society\/}
  {\bf 108} 377--428

\bibitem{Oseledec1968}
Oseledec V~I 1968 {\em Transactions of the Moscow Mathematical Society\/} {\bf
  19} 197--231

\bibitem{kingman_subadditive_1973}
Kingman J~F~C 1973 {\em Ann. Probab.\/} {\bf 1} 883--899

\bibitem{protasov2013lower}
Protasov V~Y and Jungers R~M 2013 {\em Linear Algebra and its Applications\/}
  {\bf 438} 4448--4468

\bibitem{key1990lower}
Key E~S 1990 {\em Journal of Theoretical Probability\/} {\bf 3} 477--488

\bibitem{key1987computable}
Key E 1987 {\em Probability Theory and Related Fields\/} {\bf 75} 97--107

\bibitem{pincus1985strong}
Pincus S 1985 {\em Transactions of the American Mathematical Society\/} {\bf
  287} 65--89

\bibitem{Chassain1984}
Chassaing P, Letac G and Mora M 1984 {B}rocot sequences and random walks in
  $\mathrm{SL}(2,\mathbb{R})$ {\em Probability measures on groups {VII}\/} ed
  Heyer H (Berlin: Springer) pp 36--48

\bibitem{viswanath_random_2000}
Viswanath D 2000 {\em Mathematics of Computation of the American Mathematical
  Society\/} {\bf 69} 1131--1155

\bibitem{janvresse_how_2007}
Janvresse {\'E}, Rittaud B and {de la Rue} T 2007 {\em Probability Theory and
  Related Fields\/} {\bf 142} 619--648

\bibitem{lima_exact_1994}
Lima R and Rahibe M 1994 {\em Journal of Physics A: Mathematical and General\/}
  {\bf 27} 3427--3437

\bibitem{cook_lyapunov_1990}
Cook J and Derrida B 1990 {\em Journal of Statistical Physics\/} {\bf 61}
  961--986

\bibitem{mannion_products_1993}
Mannion D 1993 {\em The Annals of Applied Probability\/} {\bf 3} 1189--1218

\bibitem{marklof_explicit_2008}
Marklof J, Tourigny Y and Wołowski L 2008 {\em Trans. Amer. Math. Soc.\/} {\bf
  360} 3391--3427

\bibitem{pollicott_maximal_2010}
Pollicott M 2010 {\em Invent. math.\/} {\bf 181} 209--226

\bibitem{DAlessandro1999}
{D'Alessandro} D, Dahleh M and Mezi\'{c} I 1999 {\em IEEE Transactions on
  Automatic Control\/} {\bf 44} 1852--1863

\bibitem{stroock2002chaotic}
Stroock A~D, Dertinger S~K, Ajdari A, Mezi{\'c} I, Stone H~A and Whitesides G~M
  2002 {\em Science\/} {\bf 295} 647--651

\bibitem{khakhar1987case}
Khakhar D, Franjione J and Ottino J 1987 {\em Chemical Engineering Science\/}
  {\bf 42} 2909--2926

\bibitem{aref1984stirring}
Aref H 1984 {\em Journal of fluid mechanics\/} {\bf 143} 1--21

\bibitem{Boyland2000}
Boyland P~L, Aref H and Stremler M~A 2000 {\em Journal of Fluid Mechanics\/}
  {\bf 403} 277--304

\bibitem{Thiffeault2006}
Thiffeault J~L and Finn M~D 2006 {\em Philosophical Transactions of the Royal
  Society of London {A}\/} {\bf 364} 3251--3266

\bibitem{Finn2011}
Finn M~D and Thiffeault J~L 2011 {\em {SIAM} Review\/} {\bf 53} 723--743

\bibitem{ottino1989kinematics}
Ottino J~M 1989 {\em The kinematics of mixing: stretching, chaos, and
  transport\/} vol~3 (Cambridge University Press)

\bibitem{sturman2006mathematical}
Sturman R, Ottino J~M and Wiggins S 2006 {\em The mathematical foundations of
  mixing: the linked twist map as a paradigm in applications: micro to macro,
  fluids to solids\/} vol~22 (Cambridge University Press)

\bibitem{Antonsen1996}
{Antonsen, Jr} T~M, Fan Z, Ott E and Garcia-Lopez E 1996 {\em Physics of
  Fluids\/} {\bf 8} 3094--3104

\bibitem{Haynes2005}
Haynes P~H and Vanneste J 2005 {\em Physics of Fluids\/} {\bf 17} 097103

\bibitem{ThiffeaultAosta2004}
Thiffeault J~L 2008 Scalar decay in chaotic mixing {\em Transport and Mixing in
  Geophysical Flows\/} ({\em Lecture Notes in Physics\/} vol 744) ed Weiss J~B
  and Provenzale A (Berlin: Springer) pp 3--35

\bibitem{crisanti_generalized_1988}
Crisanti A, Paladin G and Vulpiani A 1988 {\em Journal of Statistical
  Physics\/} {\bf 53} 583--601

\bibitem{vanneste2010estimating}
Vanneste J 2010 {\em Physical Review E\/} {\bf 81} 036701

\bibitem{przytycki1983ergodicity}
Przytycki F 1983 {\em Annales scientifiques de l'Ecole normale
  sup{\'e}rieure\/} {\bf 16} 345--354

\bibitem{parker2012practical}
Parker T~S and Chua L 2012 {\em Practical numerical algorithms for chaotic
  systems\/} (Springer Science \& Business Media)

\bibitem{ayyer_exponential_2007}
Ayyer A and Stenlund M 2007 {\em Chaos: An Interdisciplinary Journal of
  Nonlinear Science\/} {\bf 17} 043116

\end{thebibliography}

\end{document}